\documentclass[12pt]{article}

\usepackage{amssymb}
\usepackage{cmupint}
\usepackage{geometry}
\usepackage{parskip} 
\usepackage{hyperref}
\usepackage{mathtools}

\usepackage[shortlabels]{enumitem}

\usepackage{lmodern}
\usepackage[T1]{fontenc}

\DeclareSymbolFont{greeksymbols}{U}{jkpssmia}{m}{it}
\DeclareMathSymbol{\lambda}{\mathalpha}{greeksymbols}{21}
\DeclareMathSymbol{\sigma}{\mathalpha}{greeksymbols}{27}
\DeclareMathSymbol{\tau}{\mathalpha}{greeksymbols}{28}

\DeclareFontFamily{U}{matha}{\hyphenchar\font45}
\DeclareFontShape{U}{matha}{m}{n}{
      <5> <6> <7> <8> <9> <10> gen * matha
      <10.95> matha10 <12> <14.4> <17.28> <20.74> <24.88> matha12
      }{}
\DeclareSymbolFont{matha}{U}{matha}{m}{n}

\DeclareMathSymbol{\cap}           {2}{matha}{"58}
\DeclareMathSymbol{\cup}           {2}{matha}{"59}

\DeclareMathSymbol{\sim}           {3}{matha}{"12}

\DeclareFontFamily{U}{mathx}{\hyphenchar\font45}
\DeclareFontShape{U}{mathx}{m}{n}{
      <5> <6> <7> <8> <9> <10>
      <10.95> <12> <14.4> <17.28> <20.74> <24.88>
      mathx10
      }{}
\DeclareSymbolFont{mathx}{U}{mathx}{m}{n}

\DeclareMathDelimiter{(}              {matha}{"70}{mathx}{"00}
\DeclareMathDelimiter{)}              {matha}{"71}{mathx}{"08}

\DeclareMathDelimiter{\langle}     {4}{matha}{"78}{mathx}{"40}
\DeclareMathDelimiter{\rangle}     {5}{matha}{"79}{mathx}{"44}

\DeclareMathDelimiter{\lbrace}     {4}{matha}{"74}{mathx}{"20}
\DeclareMathDelimiter{\rbrace}     {5}{matha}{"75}{mathx}{"28}

\usepackage{tikz}
\newcommand{\osum}{ 
  \mathop{
    \mathchoice
      {\buildosum{\displaystyle}{0.1}}
      {\buildosum{\textstyle}{0.075}}
      {\buildosum{\scriptstyle}{0.075}}
      {\buildosum{\scriptscriptstyle}{0.075}}
  }\displaylimits 
}

\newcommand\buildosum[2]{%
  \begin{tikzpicture}[baseline=(char.base), inner sep=0, outer sep=0]
    \draw (-0.2ex,0) circle (#2); 
    \node (char) at (0,0) {$#1\sum$};
  \end{tikzpicture}%
}

\usepackage[scr=kp]{mathalpha}
\usepackage[sans]{dsfont}

\everymath{\displaystyle}

\makeatletter
\AtBeginDocument{
  \check@mathfonts
  \setlength{\fontdimen17\textfont2}{\fontdimen16\textfont2}
}
\makeatother

\usepackage{contour}
\usepackage[normalem]{ulem}

\contourlength{0.9pt}

\newcommand{\myuline}[1]{%
  \uline{\phantom{#1}}%
  \llap{\contour{white}{#1}}%
}

\let\leq\leqslant
\let\geq\geqslant

\let\epsilon\varepsilon
\let\theta\vartheta
\renewcommand{\P}{\mathds{P}}

\newcommand{\C}{\mathbb{C}}
\newcommand{\N}{\mathbb{N}}
\newcommand{\R}{\mathds{R}}
\newcommand{\sphere}{\mathbb{S}}

\newcommand{\F}{\mathscr{F}}
\newcommand{\M}{\mathcal{M}}

\newcommand{\B}{\mathsf{B}}
\newcommand{\sfE}{\mathsf{N}}
\newcommand{\sfF}{\mathsf{F}}

\newcommand{\sfL}{\mathsf{L}}
\newcommand{\sfP}{\mathsf{P}}

\newcommand{\frakG}{\mathfrak{G}}

\newcommand{\PS}{\mathsf{PS}}
\newcommand{\regPS}{\mathsf{regPS}}
\newcommand{\schwartz}{\mathsf{S}}

\newcommand{\E}{\mathscr{E}}

\newcommand{\scrA}{\mathscr{A}}
\newcommand{\scrD}{\mathscr{D}}
\newcommand{\scrZ}{\mathscr{Z}}

\newcommand{\AAS}{\mathsf{AAS}}
\newcommand{\UCB}{\mathsf{UCB}}

\let\silcrow\S
\renewcommand{\S}{$\silcrow$}

\newcommand{\dualradon}{\mathcal{R}^*}

\renewcommand{\Re}{\mathsf{Re}}

\DeclareMathOperator{\cis}{cis}

\DeclareMathOperator{\Si}{Si}
\newcommand{\stirlingi}{\genfrac{[}{]}{0pt}{}}

\makeatletter
\newcommand{\superimpose}[2]{{%
  \ooalign{%
    \hfil$\m@th#1\@firstoftwo#2$\hfil\cr
    \hfil$\m@th#1\@secondoftwo#2$\hfil\cr
  }%
}}
\makeatother

\newcommand{\asymplim}{\mathrel{\mathpalette\superimpose{{\asymp}{-}}}}

\usepackage{amsthm}

\theoremstyle{definition}

\newtheorem*{theorem*}{Theorem}

\newtheorem{theorem}{Theorem}[section]
\newtheorem{lem}[theorem]{Lemma}
\newtheorem{prop}[theorem]{Proposition}
\newtheorem{cor}[theorem]{Corollary}

\newtheorem{defn}[theorem]{Definition}

\theoremstyle{remark}

\newtheorem*{rem}{Remark}

\title{Does the Barron space really defy \\ the curse of dimensionality?}
\author{Olov Schavemaker}
\date{}

\begin{document}
\maketitle
\raggedright

\section{Introduction}\label{sec:intro}

Let $\F\!$ denote the Fourier transform (on Schwartz distributions).

The Barron space $\B^1=\{\F\{\phi\}:\phi,\lvert\cdot\rvert\phi\in\ell_1(\R^n)\}$ has been heralded as a set of functions that can be efficiently approximated by shallow neural networks, i.e., functions of the form
\begin{align*}
\R^n\ni x\mapsto\sum_{j=1}^m a_j\rho(\langle w_j\mid x\rangle-b_j)\qquad(a_j\in\C,b_j\in\R,\rho:\R\to\R,w_j\in\R^n)
\end{align*}
in the sense that shallow neural networks of the above form can approximate Barron functions with an accuracy rate of $m^{-1/2}$ \cite{felix,barron}; the exponent of $m$ \\ being independent of $n$ suggests that the Barron space may defy the curse of dimensionality, a piece of mathematical folklore which says that accuracy rates \\ are often of the form
\begin{align*}
(\text{number of terms})^{-(\text{smoothness})/(\text{dimension})}
\end{align*}
which looks like $m^{-(\text{smoothness})/n}$ in the case of the Barron space. So, unless the Barron space has approximately $n/2$ ``units of smoothness'', it defies the curse \\ of dimensionality. Barron functions indeed need not have anywhere near $n/2$ \\ units of classical smoothness (an example is given in Appendix \ref{sec:barron_example}) so they do \\ defy the curse of dimensionality where classical smoothness is concerned, but \\ are we sure classical smoothness is the most appropriate notion of smoothness when it comes to (shallow) neural networks?

A remark before we continue. The term ``Barron space'' is ambiguous within the literature and may refer to either of two related yet distinct function spaces: the one we refer to as the Barron space or a space $\sfE(\phi)$ of functions $f:U\to\R$ (with $U\subset\R^n$ nonempty) of the form
\begin{center}
$f(x)=\textstyle\int\displaystyle a\cdot\phi(\langle w\mid x\rangle+c)\,\mu(da,dw,dc)$ where $\phi:\R\to\R$ \\ and $\mu$ is a Borel probability measure on $\R\times\R^n\times\R$
\end{center}
which may be thought of as a shallow neural network of infinite width \cite{felix}.

In this paper, we introduce a new kind of smoothness spaces based on a decom- position via spherical harmonics and differential operators defined as multiplier operators of a distributional Mellin transform. We call these spaces ADZ spaces, and use them to demonstrate that the Barron space does not defy the curse of dimensionality in the sense that $\B^1\subset\sfL^{\!1}$ where $\sfL^{\!1}$ is one of the ADZ spaces.

The goal of the author in writing this paper was primarily to motivate the idea that classical smoothness may not be the most appropriate notion of smoothness \\ of approximands when we are concerned with approximating them with shallow neural networks. Our main result (Theorem \ref{thm:main}) perfectly illustrates our main takeaway that the Barron space possesses nonclassical smoothness which is con- ducive to being approximable by shallow neural networks; Theorem \ref{thm:main} reads as follows: for $\alpha\in\N(\not\ni 0),$
\begin{center}
$\B^\alpha\subset\sfL^{\!\alpha}\subset\sfE^\alpha$ with $\lVert{}\cdot{}\!\mid\sfE^\alpha\rVert\leq\lVert{}\cdot{}\!\mid\sfL^{\!\alpha}\rVert\leq\lVert{}\cdot{}\!\mid\B^\alpha\rVert,$
\end{center}
where $\B^\alpha=\{\F\{\phi\}:\phi,\lvert\cdot\rvert^\alpha\phi\in\ell_1(\R^n)\},$ (obviously a generalization of the Barron space), $\sfL^{\!\alpha}$ is the aforementioned ADZ space and $\sfE^\alpha$ is defined in Definition \ref{def:n_alpha}.

In {\S\,}\ref{sec:main_NN} we investigate the probability that $f\in\sfE^\alpha$ (with $\alpha>1)$ is within $\epsilon$ (uni- fomly on a compactum) of a shallow neural network \myuline{with random inner weights} \myuline{and biases} (most commonly dubbed a Random Vector Functional Link network (RVFL) in the literature). We are not the first ones to investigate what sorts of functions RVFLs can approximate or/and how well they can approximate them \cite{bresler,needell,hsu}.

While the inner weights and biases have nice distributions in \cite{bresler} (Gau{\ss}ian inner weights and uniform biases all independent of each other), we follow \cite{needell} in the sense that the distribution we draw our inner weights and biases from depends \\ on the approximand, and may thus be quite ``ugly'' for some approximands; this makes sense however, as we allow more approximands than \cite{bresler}. We get back to this in {\S\,}\ref{sec:main_NN}.

While our main result concerns $\alpha\in\N,$ the result can be adapted to $\alpha=0$ with suitable modifications (Theorem \ref{thm:main_base_case}). We then get that $\B\subset\sfL\subset\sfE$ with $\lVert\cdot\rVert_\B\leq$ $\lVert\cdot\rVert_\sfL\leq\lVert\cdot\rVert_\sfE$ where $\B=\F\{\ell_1(\R^n)\}$ is the so-called Wiener algebra and $\sfE$ is the image of a mixed Lebesgue space under the dual Radon transform; see {\S\,}\ref{sec:main_statement} for more details.

We know of three papers, i.e., \cite{radonNN,felix,liflyand}, whose results or/and ideas may be likened to ours. In \cite{radonNN} shallow neural networks are related to the dual Radon transform too. In \cite{felix} it is shown that $\B^\alpha\subset\sfE(\delta^{(-\alpha)})$ for $\alpha=1,2,$ where $\delta^{(-\alpha)}$ \\ is defined in {\S\,}\ref{sec:conventions}. Note that $\sfE^\alpha$ is decently similar to $\sfE(\delta^{(-\alpha)}).$ In \cite{felix} it is also shown that $\B^1\not\subset\sfE(\delta^{(-2)})\subset\sfE(\delta^{(-1)}).$ The theorem in \cite{liflyand} exposes some of the smoothness hiding within the Wiener algebra $\B,$ as does our Theorem \ref{thm:main_base_case}. See \\ {\S\,}\ref{sec:main_discussion} for more detailed comparisons.

In {\S\,}\ref{sec:conventions}, we gather some conventions we will use throughout the paper. In {\S\,}\ref{sec:prelims}, we gather some theory/concepts that will be used in the sections thereafter. In {\S\,}\ref{sec:adz} \\ we define $\B^\alpha$ \& $\B$ and discuss some aspects of their definitions. Section \ref{sec:main_base_case} is dedi- cated to stating and proving Theorem \ref{thm:main_base_case}. Section \ref{sec:main} consists of three parts: first we state and prove our main theorem, Theorem \ref{thm:main}, leveraging Theorem \ref{thm:main_base_case} and \\ its proof, then we connect with RVFLs in {\S\,}\ref{sec:main_NN}, and lastly we discuss our results and compare them to extant results in detail in {\S\,}\ref{sec:main_discussion}. In the appendices we have gathered proofs too long to include in the sections of their statements.

\section{Conventions}\label{sec:conventions}

We gather here conventions we use throughout the paper. In no particular order,
\begin{itemize}
\item $\cis_\ell\coloneq\begin{dcases*}
\cos & if $\ell$ is even \\
i\sin & if $\ell$ is odd
\end{dcases*}$
\item All function spaces will be over $\R$ or $\R^n$ for some $n\geq2$ unless indicated otherwise. So $\ell_1$ means either $\ell_1(\R)$ or $\ell_1(\R^n).$
\item by default we consider $(0,\infty)$ to be equipped with the measure $dt/t$ so that $\ell_1(0,\infty)=\ell_1(dt/t)$
\item ``$\overset{\text{F}}{=}$'' symbolizes equality due to Fubini's theorem
\item every function is complex-valued in general
\item the Pochhammer symbol, $(a)_k,$ denotes the rising factorial
\item $\N_0=\N\cup\{0\}$ and $\N$ does not include 0
\item underlined integration elements like $\underline{d\alpha}$ or $\underline{d\theta}$ are w.r.t.\ the surface area measure on $\sphere^{n-1}$ (so $\lvert\sphere^{n-1}\rvert=\textstyle\int\displaystyle\underline{d\alpha}).$
\item $a\lesssim b$ means that $a\leq bc$ for some $c>0;$ we may write $\lesssim_n$ or $\lesssim_{n,\ell}$ to specify that the hidden constant depends only on $n$ or $n,\ell$ etc.
\item $C$ denotes the space of continuous functions
\item $\UCB$ denotes the space of uniformly continuous, bounded functions
\item $\partial$ is always w.r.t.\ the unnamed variable or $t$
\item $\hat{u}\coloneq u/\lvert u\rvert$ for $u\neq0$
\item $C_0$ comprises those continuous functions which vanish at infinity
\item by $\int^\to f$ and $\int_a f$ we mean $x\mapsto\lim_{m\to\infty}\int_x^m f$ and $x\mapsto\int_a^x f$ resp.
\item $K(r)\coloneq\{x\in\R^n:\lvert x\rvert\leq r\}$
\item $\delta^{(-\alpha)}:\R\ni b\mapsto b_+^{\alpha-1}/(\alpha-1)!$ for $\alpha>1$ and $\delta^{(-1)}(b)=
\begin{dcases*}
1&if $b>0$\\
0&if $b<0$
\end{dcases*}
$
\item square brackets may denote Iverson brackets
\item $L\coloneq\{2,4,6,8,\ldots\}$
\item $[m]\coloneq\{1,\ldots,m\}$
\item if $\psi$ is Lipschitz continuous, $\mathrm{Lip}(\psi)$ denotes its Lipschitz constant
\item $\%$ denotes the modulo operation
\item the space of locally integrable functions will be denoted $\ell_1^\text{loc}$
\item $\Delta$ symbolizes the Laplacian on $\R^n$
\end{itemize}
Lastly, we use the asymptotic notation introduced in \cite[pp.\ 2f.]{hardy}; i.e.,
\begin{itemize}
\item $f\asymplim\phi$ means that $\lim f/\phi$ exists
\item $f\sim\phi$ means that $\lim f/\phi=1$
\item $f\preccurlyeq\phi$ means that $f=O(\phi)$
\end{itemize}
Since $\sim$ is already taken, we use $\propto$ to denote proportionality. As with $\lesssim$ we may write $\preccurlyeq_n$ or $\propto_{\ell,n}$ to indicate the hidden constant depends only on $n$ or $\ell,n$ etc.

\section{Preliminaries}\label{sec:prelims}

In this section we introduce (and motivate) some theory/concepts that we shall need later.

\subsection{Stirling numbers of the first kind}

Let $\alpha,m\in\N_0$ in this subsection.

The \myuline{unsigned Stirling numbers of the first kind} $\textstyle\stirlingi{\alpha}{m}$ count how many permutations of $\{1,\ldots,\alpha\}$ have exactly $m$ cycles. They satisfy the following recurrence relation:
\begin{align*}
\stirlingi{\alpha+1}{m}=\alpha\stirlingi{\alpha}{m}+\stirlingi{\alpha}{m-1}
\end{align*}
The \myuline{signed Stirling numbers of the first kind} may be defined as follows:
\begin{align*}
s(\alpha,m)=(-1)^{\alpha-m}\stirlingi{\alpha}{m}
\end{align*}
They satisfy the recurrence relation $s(\alpha+1,m)=s(\alpha,m-1)-\alpha s(\alpha,m).$

More information about them may be found in, e.g., \cite[{\S\,}26.8]{dlmf}.

\subsection{Distributions}\label{sec:fungkang}

\begin{defn}
$O_M\coloneq\{N\in C^\infty:\omega\partial^\alpha N\text{ is bounded for all }\alpha\in\N_0\text{ and }\omega\in\schwartz\}.$
\end{defn}
The function space $O_M$ is characterized by the following fact \cite[Thm 4.11.2]{horvath}:
\begin{prop}\label{prop:O_M}
$N\in O_M\Rightarrow N\psi\in\schwartz'$ for all $\psi\in\schwartz'$
\end{prop}
Let $\schwartz$ and $\schwartz'$ denote the familiar spaces of Schwartz functions and distributions on the real line resp. Now consider the logarithmic/exponential substitution
\begin{align*}
\E^{-1}:\schwartz\ni\omega\mapsto\omega(\ln t)/t\qquad\text{as a function of }t\in(0,\infty)
\end{align*}
We define $\PS\coloneq\E^{-1}\{\schwartz\}$ so that $\E^{-1}:\schwartz\to\PS$ is an isomorphism. Its inverse is plainly $\E:\PS\ni\omega\mapsto e^y\omega(e^y),$ where $\R\ni y\mapsto e^y\omega(e^y)$ is a Schwartz function.

Let $\PS'$ be the space of continuous linear functionals on $\PS.$ We now define
\begin{center}
$\E:\PS'\to\schwartz'$ by letting $\langle\E\{\psi\}\mid\omega\rangle=\langle\psi\mid\E^{-1}\{\omega\}\rangle$ where $\psi\in\PS'$ and $\omega\in\schwartz$

$\E^{-1}:\schwartz'\to\PS'$ by letting $\langle\E^{-1}\{\psi\}\mid\omega\rangle=\langle\psi\mid\E\{\omega\}\rangle$ where $\psi\in\schwartz'$ and $\omega\in\PS$
\end{center}
Plainly $\E$ and $\E^{-1}$ are linear and sequentially continuous inverses of each other (and therefore bijections). If $\psi\in\ell_1^\text{loc}\cap\schwartz'$ and $\omega\in\PS,$ then substituting $t=e^y$ yields that
\begin{align*}
\langle\E^{-1}\{\psi\}\mid\omega\rangle&=\langle\psi\mid\E\{\omega\}\rangle=\int\psi(y)\E\{\omega\}(y)\,dy=\int\psi(y)e^y\omega(e^y)\,dy\\
&=\int_0^\infty\psi(\ln t)\omega(t)\,dt
\end{align*}
Since $\psi\in\ell_1^\text{loc}\cap\schwartz'$ is conventionally identified with $\schwartz\ni\omega\mapsto\textstyle\int\psi(y)\omega(y)\,dy,$ it thus makes sense to identify $\psi(\ln t)\in\ell_1^\text{loc}(dt/t)$ with $\PS\ni\omega\mapsto\textstyle\int_0^\infty\psi(\ln t)\omega(t)\,dt.$ Ergo, $\E^{-1}\{\psi\}[t]=\psi(\ln t)$ if $\psi\in\ell_1^\text{loc}\cap\schwartz';$ we wrote ``$[t]$'' to indicate that we identify the distribution preceding the variable in square brackets with the function with said dummy variable on the other side of the equal sign.

Conversely, if $\psi\in\ell_1^\text{loc}(0,\infty)\cap\PS'$ and $\omega\in\schwartz,$ substituting $y=\ln t$ yields that
\begin{align*}
\langle\E\{\psi\}\mid\omega\rangle&=\langle\psi\mid\E^{-1}\{\omega\}\rangle=\int_0^\infty\psi(t)\E^{-1}\{\omega\}(t)\,dt=\int_0^\infty\psi(t)\omega(\ln t)\,\frac{dt}{t}\\
&=\int\psi(e^y)\omega(y)\,dy=\langle\psi(e^y)\mid\omega(y)\rangle
\end{align*}
so $\E\{\psi\}=\psi(e^y)\in\ell_1^\text{loc}(dy)$ if $\psi\in\ell_1^\text{loc}(0,\infty)\cap\PS'\eqcolon\regPS'.$

For more details on $\PS',$ see \cite{fungkang}.

Returning to $\schwartz'$ for a second, let us recall that $\F\{\psi\},$ the Fourier transform of $\psi,$ may be defined for any $\psi\in\schwartz'$ by means of \cite[Definition 14.23]{kolk}
\begin{align*}
\langle\F\{\psi\}\mid\omega\rangle=\langle\psi\mid\F\{\omega\}\rangle
\end{align*}
where $\omega\in\schwartz$ and $\F\{\phi\}(x)\coloneq\textstyle\int\displaystyle\phi(u)\exp(i\langle u\mid x\rangle)\,du$ for all $\phi\in\ell_1.$ It is well-known that $\F:\schwartz'\to\schwartz'$ defined thusly is a linear isomorphism that extends the Fourier transform on $\ell_1(\R)\supset\schwartz$ \cite[Thm 14.24]{kolk}. Do note that our normalization of the Fourier transform is different to many of the sources we cite.

In addition to $\schwartz'$ and $\PS'$ we shall also need the space of compactly supported $C^\infty(0,\infty)$ functions and the usual space of continuous linear functionals thereon, which we style as $\sfP$ and $\sfP'$ resp. There exists a continuous restriction mapping $\iota:\PS'\to\sfP'$ \cite[pg.\ 585]{fungkang}.

One of the nice properties of $\schwartz'$ is that it is closed under derivatives. An analogous property holds for $\PS'.$ Indeed, \cite[Prop.\ 1(i)]{fungkang} allows us to deduce that $\PS'$ is closed under $t\partial.$ We may therefore conclude that
\begin{lem}\label{lem:stirling}
$\PS'$ is closed under $t^\alpha\partial^{\alpha-1}t^{-1}=\textstyle\sum_{m=1}^\alpha\displaystyle s(\alpha,m)(t\partial)^{m-1}$ for all $\alpha\in\N$
\end{lem}
\begin{proof}
By induction. The case $\alpha=1$ is trivial. By the recurrence relation the $s(\alpha,m)$ satisfy (and the identities $s(\alpha,\alpha+1)=0$ and $s(\alpha,0)=0)$:
\begin{align*}
\sum_{m=1}^{\alpha+1}s(\alpha+1,m)(t\partial)^{m-1}&=\sum_{m=1}^{\alpha+1}s(\alpha,m-1)(t\partial)^{m-1}-\alpha\sum_{m=1}^{\alpha+1}s(\alpha,m)(t\partial)^{m-1}\\
&=\sum_{m=0}^\alpha s(\alpha,m)(t\partial)^m-\alpha\sum_{m=1}^\alpha s(\alpha,m)(t\partial)^{m-1}\\
&=\sum_{m=1}^\alpha s(\alpha,m)(t\partial)^m-\alpha t^\alpha\partial^{\alpha-1}t^{-1}\\
&=t\partial(t^\alpha\partial^{\alpha-1}t^{-1})-\alpha t^\alpha\partial^{\alpha-1}t^{-1}=t^{\alpha+1}\partial^\alpha t^{-1}
\end{align*}
by the product rule.
\end{proof}
\begin{cor}\label{cor:stirling}
$\PS'$ is closed under $t^\alpha\partial^\alpha=\textstyle\sum_{m=0}^\alpha s(\alpha,m)(t\partial)^m$ for every $\alpha\in\N_0$
\end{cor}
\begin{proof}
If $\alpha\in\N,$ multiply both sides of $t^\alpha\partial^{\alpha-1}t^{-1}=\textstyle\sum_{m=1}^\alpha\displaystyle s(\alpha,m)(t\partial)^{m-1}$ on the right by $t\partial$ noting that $s(\alpha,0)=0.$ If $\alpha=0,$ the claim is trivial.
\end{proof}
If $\sfP'$ is closed under $t\partial$ as well, Lemma \ref{lem:stirling} and Corollary \ref{cor:stirling} are true on $\sfP'$ too. Actually, more is true. Unlike $\PS',$ it is the case that $\sfP'$ is closed under $t$ and $\partial$ separately. Indeed, let $\psi\in\sfP'$ and $\omega\in\sfP.$ Then $\partial\omega\in\sfP,$ and hence $\partial\psi$ may be defined by $\langle\partial\psi\mid\omega\rangle=\langle\psi\mid-\partial\omega\rangle,$ where $\langle\psi\mid\omega\rangle\coloneq\textstyle\int_0^\infty\displaystyle\psi\omega$ if $\psi\in\ell_1^\text{loc}(0,\infty),$ in conformity with integration by parts. If $M\in C^\infty(0,\infty),$ we can define $M\psi$ by $\langle M\psi\mid\omega\rangle=\langle\psi\mid M\omega\rangle$ because $M\omega\in\sfP$ again. Because the identity function is smooth, $\sfP'$ is closed under multiplication by $t$ in particular.

In $\sfP'$ it is still true that derivatives are ``uglier'' than the original:
\begin{lem}\label{lem:distributional_antiderivative}
If $G\in\sfP'$ and $\partial G\in\ell_1^\text{loc}(0,\infty),$ then $G\in C(0,\infty).$
\end{lem}
\begin{proof}
Let $\frakG\coloneq\textstyle\int_1\displaystyle\partial G(z)\,dz.$ Then $\frakG$ is (absolutely) continuous (since $\{\partial G\}$ is uniformly integrable) and, a fortiori, a $\sfP$ distribution. Moreover, $\partial\frakG=\partial G$ by \\ the Lebesgue differentiation theorem. Since any $\omega\in\sfP$ has compact support,
\begin{align*}
\langle\partial(G-\frakG)\mid\omega\rangle=\langle\partial G\mid\omega\rangle+\langle\frakG\mid\partial\omega\rangle=\int_0^\infty\partial\frakG(t)\omega(t)\,dt+\int_0^\infty\frakG(t)\partial\omega(t)\,dt=0
\end{align*}
courtesy of integration by parts. By \cite[Thm 4.3]{kolk} $G-\frakG$ is therefore a constant function, so $G,$ being the sum of $\frakG\in C(0,\infty)$ and a constant function, is thus continuous.
\end{proof}

Insightful in its own right \cite[Prop.\ 1(i)]{fungkang} says that
\begin{center}
$\psi\in\PS'$ iff $\psi[t]=(t\partial)^\alpha G(t)$ for some $G\in C(0,\infty)$ \\ behaving polylogarithmically at both 0 and $\infty$
\end{center}
Lastly, we shed some light on what kind of functions live in $\regPS'.$

It follows from \cite[Prop.\ 1(i)]{fungkang} that any $G\in C(0,\infty)$ behaving polylogarithmically at both 0 and $\infty$ is a $\PS$ distribution. Moreover, $\ell_\infty(0,\infty)\subset\PS';$ indeed, apply $\E^{-1}$ to $\ell_\infty\subset\schwartz'$ \cite[Example 14.22]{kolk}.

\subsection{Mellin transform}\label{sec:mellin}

The Mellin transform is usually defined as
\begin{align*}
\M\{\psi\}(s)\coloneq\int_0^\infty t^s\psi(t)\,\frac{dt}{t}
\end{align*}
for $s\in\mathbb{C}$ and functions $\psi$ on $(0,\infty)$ such that the integrand is $\ell_1(dt/t).$ Letting $s=c+iy$ with $y,c\in\R$ and substituting $t=e^u$ yields that
\begin{align*}
\M\{\psi\}(s)=\int_0^\infty t^{c+iy}\psi(t)\,\frac{dt}{t}=\int e^{cu}e^{iuy}\psi(e^u)\,du=\F\{e^{cu}\psi(e^u)\}(y)
\end{align*}
which is essentially a Fourier transform. This is useful because the distributional Fourier transform is more fleshed out than the distributional Mellin transform, so we will define the distributional Mellin transform as follows:
\begin{defn}
For $c\in\R$ we define $\M_0\coloneq\F\E:\PS'\to\schwartz'$ and
\begin{align*}
\schwartz'\ni\psi\mapsto\M_c^{-1}\{\psi\}[t]\coloneq t^{-c}\iota\{(\F\E)^{-1}\{\psi\}\}[t]\in\sfP'
\end{align*}
where we used the notation introduced in {\S\,}\ref{sec:fungkang}.
\end{defn}
Note that $\M_0^{-1}\M_0=\iota.$ Both $\E:\PS'\to\schwartz'$ and $\F:\schwartz'\to\schwartz'$ are isomorphisms, so $\M_0$ is too; ergo, if $G=\M_0^{-1}\{\psi\},$ then $G$ is not just any $\sfP$ distribution, but a $\PS$ distribution, and $\psi=\M_0\{G\}.$

If $\psi\in\ell_1(0,\infty),$ then $\M_0\{\psi\}[y]=\M\{\psi\}(iy)$ as expected. Indeed, for $\omega\in\schwartz,$
\begin{align*}
\langle\M_0\{\psi\}\mid\omega\rangle&=\langle\F\{\E\{\psi\}\}\mid\omega\rangle=\langle\E\{\psi\}\mid\F\{\omega\}\rangle=\langle\psi\mid\E^{-1}\{\F\{\omega\}\}\rangle\\
&=\int_0^\infty\psi(t)\E^{-1}\{\F\{\omega\}\}(t)\,dt=\int_0^\infty\psi(t)\F\{\omega\}(\ln t)\,\frac{dt}{t}\\
&=\int_0^\infty\psi(t)\int e^{iy\ln t}\omega(y)\,dy\,\frac{dt}{t}\overset{\text{F}}{=}\int\omega(y)\int_0^\infty t^{iy}\psi(t)\,\frac{dt}{t}\,dy\\
&=\int\omega(y)\M\{\psi\}(iy)\,dy=\langle\M\{\psi\}(iy)\mid\omega(y)\rangle
\end{align*}

\subsection{A family of Banach spaces}\label{sec:1infty}

The following family of Banach spaces will feature prominently.
\begin{defn}
Let $a\in[-\infty,\infty).$ We write $\ell_{1,\infty}(\sphere^{n-1},(a,\infty))$ for the space of equivalence classes of jointly measurable $h:\sphere^{n-1}\times(a,\infty)\to\C$ that are almost everywhere equal such that
\begin{align*}
\lVert h\rVert_{1,\infty}\coloneq\int\,\lVert h(\theta,\cdot)\rVert_\infty\,\underline{d\theta}<\infty
\end{align*}
\end{defn}
These spaces resemble those in \cite{oberlin}. As a shorthand we may write $\ell_{1,\infty}(a).$
\begin{theorem}
$\ell_{1,\infty}(a)$ is a Banach space
\end{theorem}
\begin{proof}
$\lVert\cdot\rVert_{1,\infty}$ is patently a norm. To prove $\ell_{1,\infty}(a)$ is complete it suffices to show that every absolutely convergent series in $\ell_{1,\infty}(a)$ converges \cite[Thm 1.3.9]{megginson}. Thus we suppose that $\textstyle\sum_{k=1}^\infty\displaystyle\lVert h_k\rVert_{1,\infty}<\infty.$ By Tonelli's theorem it follows that
\begin{align*}
\int\sum_{k=1}^\infty\,\lVert h_k(\theta,\cdot)\rVert_\infty\,\underline{d\theta}=\sum_{k=1}^\infty\,\lVert h_k\rVert_{1,\infty}<\infty
\end{align*}
so $h\coloneq\textstyle\sum_{k=1}^\infty\displaystyle h_k$ exists almost everywhere. Ergo $\lVert h\rVert_{1,\infty}\leq\textstyle\sum_{k=1}^\infty\displaystyle\lVert h_k\rVert_{1,\infty}<\infty$ and
\begin{align*}
\lVert h-\sum_{k=1}^K h_k\rVert_{1,\infty}\leq\sum_{k>K}\,\lVert h_k\rVert_{1,\infty}\to 0\text{ as }K\to\infty\tag*{\qedhere}
\end{align*}
\end{proof}

\subsection{Spherical harmonics}\label{sec:zonal}

A \myuline{spherical harmonic} is the restriction of a homogeneous harmonic polynomial on $\R^n$ to the unit sphere; a spherical harmonic of degree $\ell$ being the restriction of a polynomial of degree $\ell.$ Spherical harmonics of degree $\ell$ form a finite dimensional vector space; let $N(\ell,n)$ denote its dimension. The exact value of $N(\ell,n)$ is given in \cite[(11)]{mueller}.

We henceforth consider real-valued spherical harmonics over $\R$ as in \cite{mueller}; the story of complex-valued spherical harmonics over $\C$ is analogous \cite{steinweiss}. We suitably adapt the results we cite from \cite{steinweiss} insofar necessary.

Letting $\{S_{\ell j}\}_{j=1}^{N(\ell,n)}$ be an orthonormal basis of the space of spherical harmonics of degree $\ell,$ we can define the \myuline{zonal spherical harmonic} of degree $\ell$ as follows:
\begin{align*}
Z_\ell(\theta,\alpha)\coloneq\sum_{j=1}^{\mathclap{N(\ell,n)}}S_{\ell j}(\theta)S_{\ell j}(\alpha)
\end{align*}
Zonal spherical harmonics are reproducing kernels; indeed, if $Y_\ell$ is any spherical harmonic of degree $\ell,$ then $Y_\ell=\textstyle\sum_j\displaystyle Y_{\ell j}S_{\ell j}$ with $Y_{\ell j}\coloneq\textstyle\int\displaystyle S_{\ell j}(\alpha)Y_\ell(\alpha)\,\underline{d\alpha}$ and so
\begin{align}\label{eq:zonal_reproducing}
\int Z_\ell(\theta,\alpha)Y_\ell(\alpha)\,\underline{d\alpha}=\sum_j S_{\ell j}(\theta)\int S_{\ell j}(\alpha)Y_\ell(\alpha)\,\underline{d\alpha}=\sum_j S_{\ell j}(\theta)Y_{\ell j}=Y_\ell(\theta)
\end{align}
Spherical harmonics with different degrees are independent w.r.t.\ the $\ell_2$ inner product. Some other useful facts concerning the $S_{\ell j}$ and $Z_\ell$ include that
\begin{align}\label{eq:slj_matrix}
\exists\alpha_1,\ldots,\alpha_N:\
\begin{Bmatrix}
S_{\ell,1}(\alpha_1)&\cdots&S_{\ell N}(\alpha_1)\\
\vdots&\ddots&\vdots\\
S_{\ell,1}(\alpha_N)&\cdots&S_{\ell N}(\alpha_N)
\end{Bmatrix}
\text{ is invertible, where }N=N(\ell,n)
\end{align}
\cite[Lemma 6]{mueller} and
\begin{lem}\label{lem:z_bound}
$\lVert Z_\ell\rVert_1\preccurlyeq_n\ell^{\mkern1mu n-2}$ as $\ell\to\infty$
\end{lem}
\begin{proof}
It follows from \cite[Cor.\ 2.9(c)]{steinweiss} that $\lVert Z_\ell\rVert_1\lesssim_n\lVert Z_\ell\rVert_\infty\lesssim_n a_\ell$ where
\begin{align*}
a_\ell=\frac{n+2\ell-2}{\ell}\binom{n+\ell-3}{\ell-1}
\end{align*}
for $\ell\geq 2$ \cite[pg.\ 140, pg.\ 145]{steinweiss}. Since $a_\ell\sim 2\ell^{\mkern1mu n-2}$ as $\ell\to\infty$ the lemma follows.
\end{proof}
In light of \cite[Thm IV.2.10]{steinweiss} the \myuline{Poisson kernel} on $\sphere^{n-1}$ may be defined by
\begin{align*}
p(\alpha,R\theta)\coloneq\sum_{\ell=0}^\infty R^\ell Z_\ell(\theta,\alpha)
\end{align*}
where $\alpha,\theta\in\sphere^{n-1}$ and $R\in[0,1).$ To shed some light on the properties of the Poisson kernel, we paraphrase Thm II.1.9 and part of Thm II.1.10 from \cite{steinweiss}:
\begin{prop}[II.1.9]\label{prop:ii_1_9}
(a) $p\geq 0;$ (b) $\textstyle\int\displaystyle p(\alpha,x)\,\underline{d\alpha}=1$ for all $\lvert x\rvert<1;$ (c)
\begin{center}
$\int_{\mathclap{\lvert\theta-\alpha\rvert>\delta}}p(\alpha,R\theta)\,\underline{d\alpha}\to 0$ as $R\uparrow 1$ uniformly in $\theta$ for each $\delta>0$
\end{center}
\end{prop}
\begin{prop}[II.1.10]\label{prop:ii_1_10}
If $f$ is a continuous function on $\sphere^{n-1}$ then
\begin{align*}
u(x)\coloneq\int f(\alpha)p(\alpha,x)\,\underline{d\alpha}
\end{align*}
whenever $\lvert x\rvert<1$ and $u\coloneq f$ on $\sphere^{n-1}$ defines a continuous function on $K(1).$
\end{prop}
Lastly, we return to the $Z_\ell.$ By \cite[Lemma IV.2.8(a)]{steinweiss} and \cite[Thm 2]{mueller}:
\begin{align}\label{eq:z_propto_c}
Z_\ell(\theta,\alpha)\propto_{\ell,n} C^{(n-2)/2}_\ell(\langle\theta\mid\alpha\rangle)
\end{align}
where the Gegenbauer polynomials $C_\ell^\alpha(v)$ are the (unique) system of orthogonal polynomials on the interval $[-1,1]$ w.r.t.\ the weight function $(1-v^2)^{\alpha-1/2}$ such that $C_\ell^\alpha$ is of degree $\ell$ and $C_\ell^\alpha(1)=1$ \cite[pg.\ 16]{mueller}. The Gegenbauer polynomials have the nice property that $C_\ell^\alpha$ is odd (even) for odd (even) $\ell$ \cite[Thm 1]{mueller}. An \\ even nicer property of the Gegenbauer polynomials is \cite[Lemma 12]{mueller}
\begin{prop}\label{prop:funk_hecke_gegenbauer}
If $f$ is continuous on $[-1,1],$ then $\forall\nu,\theta\in\sphere^{n-1}:$
\begin{gather*}
\int f(\langle\alpha\mid\nu\rangle)C_\ell^{(n-2)/2}(\langle\theta\mid\alpha\rangle)\,\underline{d\alpha}=\\
C_\ell^{(n-2)/2}(\langle\theta\mid\nu\rangle)\lvert\sphere^{n-2}\rvert\int_{-1}^{+1}f(v)C_\ell^{(n-2)/2}(v)(1-v^2)^{(n-3)/2}\,dv
\end{gather*}
\end{prop}
Proposition \ref{prop:funk_hecke_gegenbauer} allows one to deduce the Funk-Hecke formula:
\begin{lem}\label{lem:funk_hecke}
If $h$ is integrable on $[-1,1],$ then, for any spherical harmonic $Y_\ell$ of degree $\ell,$
\begin{align*}
\int h(\langle\theta\mid\alpha\rangle)Y_\ell(\alpha)\,\underline{d\alpha}=Y_\ell(\theta)\lvert\sphere^{n-2}\rvert\int_{-1}^{+1}h(v)C_\ell^{(n-2)/2}(v)(1-v^2)^{(n-3)/2}\,dv
\end{align*}
\end{lem}
\begin{proof}
That this holds for continuous $f$ in place of integrable $h$ is the content of \cite[Thm 6]{mueller}. To obtain the result for integrable $h,$ take a sequence of continuous functions converging to $h$ and apply the bounded convergence theorem.
\end{proof}
Note that some of the sources we cite normalize the Gegenbauer polynomials differently.

\section{ADZ spaces}\label{sec:adz}

In this section we define the ADZ spaces. The following notation will be properly defined in {\S\,}\ref{sec:adz_def}, but we refer to the $\sfL^{\!\alpha}$ as ADZ spaces because
\begin{align*}
\lVert f\mid\sfL^{\!\alpha}\rVert=\lVert\scrA_\alpha\{\scrD^\alpha\{\scrZ \{f\}\}\}\rVert_{1,\infty}
\end{align*}
where the ADZ operators essentially do the following:
\begin{align*}
\scrZ:f&\mapsto\{f_\ell\}_{\ell=0}^\infty\\
\scrD^\alpha:\{f_\ell(\theta,\cdot)\}_{\ell=0}^\infty&\mapsto\{\text{D}_\ell^\alpha\{f_\ell(\theta,\cdot)\}\}_{\ell\notin L}\cup\{\text{\DH}_\ell^\alpha\{f_\ell(\theta,\cdot)\}\}_{\ell\in L}\\
\scrA_\alpha:\{G_\ell\}_{\ell=0}^\infty&\mapsto\{h_\ell^\alpha\}_{\ell=0}^\infty\mapsto\osum_{\ell=0}^\infty h_\ell^\alpha
\end{align*}
Basically, we decompose $f$ into pieces, differentiate each piece differently, and put the differentiated pieces back together with a slight twist. The rest of this section will be spend on fleshing out this definition.

\subsection{A family of functions}\label{sec:N}

Featuring prominently will be the following family of functions:
\begin{defn} For $\alpha\in\N_0$ if $\ell\notin L$ and $\alpha\in\N$ if $\ell\in L,$ let
\begin{align*}
N_\ell^\alpha:\R\ni y\mapsto
\begin{dcases*}
\frac{\displaystyle\Gamma(\tfrac{n-iy}{2})(-1)^\alpha(iy)_\alpha}{2\displaystyle\pi^{(n-1)/2}\Gamma(\tfrac{1-iy}{2})}&if $\ell=0$ \\
\frac{\displaystyle\Gamma(\tfrac{n+\ell-iy}{2})(-1)^\alpha(iy)_\alpha}{2\displaystyle\pi^{(n-1)/2}\Gamma(\tfrac{2-iy}{2})(-1)^{(\ell-1)/2}(\tfrac{1+iy}{2})_{(\ell-1)/2}}&if $\ell$ is odd \\
\frac{\displaystyle\Gamma(\tfrac{n+\ell-iy}{2})(-1)^{\alpha-1}(1+iy)_{\alpha-1}}{\displaystyle\pi^{(n-1)/2}\Gamma(\tfrac{1-iy}{2})(-1)^{(\ell-2)/2}(\tfrac{2+iy}{2})_{(\ell-2)/2}}&if $\ell\in L$
\end{dcases*}
\end{align*}
\end{defn}
We remind the reader that the parentheses with subscripts are Pochhammer sym- bols. Allowing removable singularities, we can express $N_\ell^\alpha$ more compactly as
\begin{align*}
N_\ell^\alpha(y)&=\frac{\Gamma(\tfrac{n+\ell-iy}{2})\Gamma(\tfrac{2-\ell-iy}{2})\Gamma(1-iy)}{2\displaystyle\pi^{(n-1)/2}\Gamma(\tfrac{2-iy}{2})\Gamma(\tfrac{1-iy}{2})\Gamma(1-iy-\alpha)}\\
&=\frac{\Gamma(\tfrac{n+\ell-iy}{2})\Gamma(\tfrac{2-\ell-iy}{2})\Gamma(-iy)}{\displaystyle\pi^{(n-1)/2}\Gamma(\tfrac{-iy}{2})\Gamma(\tfrac{1-iy}{2})\Gamma(1-iy-\alpha)}
\end{align*}
because $\Gamma(1+s)=s\Gamma(s)$ for $s\in i\R-\{0\}.$

The $N_\ell^\alpha$ are plainly smooth, so the asymptotic relation \cite[pg.\ 244]{pribitkin}
\begin{align*}
\lvert\Gamma(\sigma+it)\rvert\sim\sqrt{2\pi}e^{-\pi\lvert t\rvert/2}\lvert t\rvert^{\sigma-1/2}\qquad\text{uniformly as }\lvert t\rvert\to\infty
\end{align*}
for $\sigma$ in any compact subset of $\R$ can be used to deduce that
\begin{align}\label{eq:n_asymp}
\lvert N_\ell^\alpha(y)\rvert\sim\frac{1}{2}\lvert y\rvert^\alpha\biggl\{\frac{\lvert y\rvert}{2\pi}\biggr\}^{\!(n-1)/2}\qquad\text{if }\lvert y\rvert\to\infty
\end{align}
suggesting that
\begin{lem}\label{lem:NOM}
$N_\ell^\alpha\in O_M$ for $\alpha\in\N_0$ if $\ell\notin L$ and $\alpha\in\N$ if $\ell\in L$
\end{lem}
\begin{proof}
Technical; see Appendix \ref{sec:NOM}.
\end{proof}
\begin{cor}\label{lem:N}
$N_\ell^\alpha\psi\in\schwartz'$ for all $\psi\in\schwartz'$
\end{cor}
\begin{proof}
Immediate from Lemma \ref{lem:NOM} and Proposition \ref{prop:O_M}.
\end{proof}

\subsection{The definition}\label{sec:adz_def}

We will first define the ADZ operators and then the ADZ space.
\begin{defn}
$\scrZ:C\ni f\mapsto\Bigl(\sphere^{n-1}\ni\theta\mapsto\{f_\ell(\theta,\cdot)\}_{\ell=0}^\infty\Bigr),$ where
\begin{align*}
f_\ell(\theta,t)\coloneq\int f(t\alpha)Z_\ell(\theta,\alpha)\,\underline{d\alpha}\in C(\sphere^{n-1}\times(0,\infty))
\end{align*}
\end{defn}
As a reminder, $L=\{2,4,6,8,\ldots\}.$
\begin{defn}
Let $\alpha\in\N_0$ if $\ell\notin L$ and $\alpha\in\N$ if $\ell\in L.$ We call the function
\begin{align*}
\PS'\ni\psi\mapsto\M_\alpha^{-1}\Bigl\{N_\ell^\alpha\M_0\{\psi\}\Bigr\}\in\sfP'
\end{align*}
$\text{D}_\ell^\alpha$ if $\ell\notin L$ and $\text{\DH}_\ell^\alpha$ if $\ell\in L.$ Note that this function is well-defined in light of Lemma \ref{lem:N}.
\end{defn}
The reason for the obfuscating naming will become clear later.
\begin{defn}
$\scrD^\alpha:\PS'\supset\{\psi_\ell\}_{\ell=0}^\infty\mapsto\{\text{D}_\ell^\alpha\{\psi_\ell\}\}_{\ell\notin L}\cup\{\text{\DH}_\ell^\alpha\{\psi_\ell\}\}_{\ell\in L}\subset\sfP'$
\end{defn}
To define $\scrA_\alpha$ we need Abel summation; cf., e.g., \cite[Thm 9]{mueller}.
\begin{defn}
Let $X$ be a Banach space with norm $\lVert\cdot\rVert$ and consider a sequence $\{h_\ell\}_{\ell=0}^\infty\subset X.$ Let $\{h_\ell\}_{\ell=0}^\infty$ be \myuline{absolutely Abel summable} with \myuline{Abel sum} $h\in X$ if
\begin{enumerate}
\item $\sum_{\ell=0}^\infty R^\ell\lVert h_\ell\rVert<\infty$ for all $R\in(0,1);$ ergo $h_R\coloneq\sum_{\ell=0}^\infty R^\ell h_\ell\in X$ \cite[Thm 1.3.9]{megginson};
\item $h_R\to h$ in $X$ if $R\uparrow1.$
\end{enumerate}
In this case we write $h=\textstyle\osum_{\ell=0}^\infty h_\ell.$

We denote the set of all absolutely Abel summable $\{h_\ell\}_{\ell=0}^\infty\subset\ell_{1,\infty}(a)$ by $\AAS(a).$
\end{defn}
Note that $\{G_\ell\}_{\ell=0}^\infty\in\AAS(0)\Leftrightarrow\{h_\ell^\alpha\}_{\ell=0}^\infty\in\AAS(-\infty)$ for any $\alpha\in\N,$ where
\begin{align}\label{eq:h_in_terms_of_g}
h_\ell^\alpha(\theta,t)\coloneq
\begin{dcases*}
G_\ell(\theta,t)&if $t>0$ \\
(-1)^{\ell+\alpha}G_\ell(\theta,-t)&if $t<0$
\end{dcases*}
\end{align}
\begin{rem}
However we define $\sphere^{n-1}\ni\theta\mapsto h_\ell^\alpha(\theta,0)$ is irrelevant.
\end{rem}
\begin{defn}
Letting $\{h_\ell^\alpha\}_{\ell=0}^\infty$ be defined as in \eqref{eq:h_in_terms_of_g}, 
\begin{align*}
\scrA_\alpha:\AAS(0)\ni\{G_\ell\}_{\ell=0}^\infty\mapsto\osum_{\ell=0}^\infty h_\ell^\alpha\in\ell_{1,\infty}(-\infty)
\end{align*}
Since we tend to drop sub/superscript 0's, let $\scrA\coloneq\scrA_0.$
\end{defn}
We now define the ADZ space $\sfL^{\!\alpha}$ using the $N_\ell^\alpha$ from {\S\,}\ref{sec:N}.
\begin{defn}\label{def:adz}
The \myuline{ADZ space} $\sfL^{\!\alpha}$ comprises those $f\in C$ satisfying
\begin{enumerate}
\item $\scrZ\{f\}(\theta)\subset\PS'$ for all $\theta$
\item $\scrD^\alpha\{\scrZ\{f\}(\theta)\}\eqcolon\{G_\ell^\alpha(\theta,\cdot)\}_{\ell=0}^\infty$ satisfying $\{G_\ell^\alpha\}_{\ell=0}^\infty\in\AAS(0)$
\item $\lVert f\mid\sfL^{\!\alpha}\rVert\coloneq\lVert h^\alpha\rVert_{1,\infty}$ where $h^\alpha\coloneq\scrA_\alpha\{\{G_\ell^\alpha\}_{\ell=0}^\infty\}=\osum_{\ell=0}^\infty h_\ell^\alpha$ with
\begin{align*}
h_\ell^\alpha(\theta,t)=
\begin{dcases*}
G_\ell^\alpha(\theta,t)&if $t>0$ \\
(-1)^{\ell+\alpha}G_\ell^\alpha(\theta,-t)&if $t<0$
\end{dcases*}
\end{align*}
Note that $\{G_\ell^\alpha\}_{\ell=0}^\infty\in\AAS(0)\Leftrightarrow\{h_\ell^\alpha\}_{\ell=0}^\infty\in\AAS(-\infty)\Rightarrow h^\alpha\in\ell_{1,\infty}(-\infty)\Rightarrow$ $\lVert h^\alpha\rVert<\infty.$
\end{enumerate}
\end{defn}
Threading $\scrD^\alpha$ over $\theta,$ the following are now obvious:
\begin{align*}
\{G_\ell^\alpha\}_{\ell=0}^\infty&=\scrD^\alpha\{\scrZ\{f\}\}\\
h^\alpha&=\scrA_\alpha\{\scrD^\alpha\{\scrZ \{f\}\}\}\\
\lVert f\mid\sfL^{\!\alpha}\rVert&=\lVert\scrA_\alpha\{\scrD^\alpha\{\scrZ \{f\}\}\}\rVert_{1,\infty}
\end{align*}
Lastly, we sketch a proof for the following routine yet important result.
\begin{theorem}
$\sfL^{\!\alpha}$ is a seminormed space
\end{theorem}
\begin{proof}[Proof sketch]
$\sfL^{\!\alpha}$ is a vector space and $\scrA_\alpha\,\scrD^\alpha\scrZ:\sfL^{\!\alpha}\to$ $\ell_{1,\infty}(-\infty)$ is linear.
\end{proof}

\subsection{The case $\alpha=0$}

Up till now $\alpha\in\N.$ Unfortunately, ``$\sfL^{\!0}$'' does not make sense, for if $\ell\in L,$
\begin{align*}
\text{``}\!N_\ell^0\text{''}:y\mapsto\frac{\Gamma(\tfrac{n+\ell-iy}{2})\Gamma(\tfrac{2-\ell-iy}{2})}{2\displaystyle\pi^{(n-1)/2}\Gamma(\tfrac{2-iy}{2})\Gamma(\tfrac{1-iy}{2})}
\end{align*}
has a singularity at $y=0,$ so $\text{``}\!N_\ell^0\text{''}\notin C^\infty\supset O_M;$ i.e., we may well not be able to successfully apply $\M_0^{-1}$ in the definition of ``$\text{\DH}_\ell^0$''. However, there is a work-around which we can use to define $\sfL;$ our version of ``$\sfL^{\!0}$''.

To unburden notation, we will drop a bunch of superscripts; let $N_\ell\coloneq N_\ell^{[\ell\in L]}$ and let $\M_0^{-1}\{N_\ell\M_0\{\cdot\}\}$ be denoted $\text{D}_\ell$ for $\ell\notin L$ and $\text{\DH}_\ell$ for $\ell\in L.$ To define $\scrD$ (our version of ``$\scrD^0$'') we first need another definition:
\begin{defn}\label{def:F}
Letting $\Phi_\ell$ be any antiderivative of $v\mapsto f_\ell(v)/v,$ we let
\begin{align*}
\sfF\coloneq\biggl\{\{f_\ell\}_{\ell=0}^\infty\subset C(0,\infty):
\begin{aligned}
f_\ell&\in\PS'\text{ whenever }\ell\notin L\\
\Phi_\ell&\in\PS'\text{ and }\Phi_\ell(\infty)\text{ exists whenever }\ell\in L
\end{aligned}
\biggr\}
\end{align*}
and we set $\PS'\ni\Phi_\ell(\infty)-\Phi_\ell=-\int^\to f_\ell(v)\,\frac{dv}{v}\eqcolon F_\ell$ whenever $\ell\in L.$
\end{defn}
\begin{defn}
$\scrD:\sfF\ni\{f_\ell\}_{\ell=0}^\infty\mapsto
\{\text{D}_\ell\{f_\ell\}\}_{\ell\notin L}\cup\{\text{\DH}_\ell\{F_\ell\}\}_{\ell\in L}$
\end{defn}
Note that $G=\M_0^{-1}\{N_\ell\M_0\{\psi\}\}\Rightarrow\M_0\{G\}=N_\ell\M_0\{\psi\}$ because $G$ is not just a $\sfP$ distribution, but a $\PS$ distribution (recall our discussion of this in {\S\,}\ref{sec:mellin}).

We are now ready to define $\sfL.$
\begin{defn}\label{def:gamma_d}
$\sfL$ comprises those $f\in C$ satisfying
\begin{enumerate}
\item $\{f_\ell(\theta,\cdot)\}_{\ell=0}^\infty\in\sfF$ for all $\theta$
\item $\scrD\{\scrZ\{f\}\}\eqcolon\{G_\ell\}_{\ell=0}^\infty\in\AAS(0)$
\item $\lVert f\rVert_\sfL\coloneq\lVert h\rVert_{1,\infty}$ where $h\coloneq\scrA\{\{G_\ell\}_{\ell=0}^\infty\}=\osum_{\ell=0}^\infty h_\ell$ with
\begin{align*}
h_\ell(\theta,t)=
\begin{dcases*}
G_\ell(\theta,t)&if $t>0$ \\
(-1)^{\ell}G_\ell(\theta,-t)&if $t<0$
\end{dcases*}
\end{align*}
\end{enumerate}
\end{defn}
Note that $\sfL$ is a seminormed space just like the ADZ spaces. The reasoning is entirely analogous.

\subsection{Discussion}

At the start of {\S\,}\ref{sec:adz} we posited: ``Basically, we decompose $f$ into pieces, differentiate each piece differently, and put the differentiated pieces back together with a slight twist''. Having defined everything formally, we now evince this claim.

We first show Abel summation undoes $\scrZ$ (corroborating that we are putting the pieces back together fittingly).
\begin{theorem}\label{thm:a_undoes_z}
Let $f\in C$ and  fix $t\in(0,\infty).$ If $f_{\ell\mid t}(\theta)\coloneq f_\ell(\theta,t)$ and $f_{\cdot\mid t}(\theta)\coloneq$ $f(t\theta),$ then $\{f_{\ell\mid t}\}_{\ell=0}^\infty\subset\ell_1(\sphere^{n-1})$ is absolutely Abel summable with Abel sum $f_{\cdot\mid t}.$ 
\end{theorem}
\begin{proof}
Plainly $f_{\cdot\mid t}\in\UCB(\sphere^{n-1})\subset\ell_1(\sphere^{n-1}).$ By Lemma \ref{lem:z_bound},
\begin{align*}
\lVert f_{\ell\mid t}\rVert_1\leq\lVert Z_\ell\rVert_1\lVert f_{\cdot\mid t}\rVert_1\preccurlyeq_{n,t}\ell^{\mkern1mu n-2}
\end{align*}
Accordingly, $\textstyle\sum_{\ell=0}^\infty\displaystyle R^\ell\lVert f_{\ell\mid t}\rVert_1<\infty$ for all $R\in(0,1),$ so
\begin{align*}
\sum_{\ell=0}^\infty R^\ell f_{\ell\mid t}(\theta)=\sum_{\ell=0}^\infty R^\ell\int f_{\cdot\mid t}(\alpha)Z_\ell(\theta,\alpha)\,\underline{d\alpha}\overset{\text{F}}{=}\int f_{\cdot\mid t}(\alpha)p(\alpha,R\theta)\,\underline{d\alpha}
\end{align*}
As $f_{\cdot\mid t}\in\UCB(\sphere^{n-1}),$ following the relevant part of the proof of \cite[Thm II.1.10]{steinweiss} mutatis mutandis yields that $\textstyle\sum_{\ell=0}^\infty\displaystyle R^\ell f_{\ell\mid t}\to f_{\cdot\mid t}$ uniformly if $R\uparrow1,$ which in turn yields convergence in $\ell_1(\sphere^{n-1})$ because $\lVert{}\cdot{}\!\mid\ell_1(\sphere^{n-1})\rVert\lesssim_n\lVert{}\cdot{}\!\mid\ell_\infty(\sphere^{n-1})\rVert.$
\end{proof}
We now argue why it is reasonable to interpret the $\text{D}_\ell^\alpha$ and $\text{\DH}_\ell^\alpha$ as differential operators. Since they are virtually identical, we focus on the $\text{D}_\ell^\alpha.$

It is well-known that $\psi\mapsto\F^{-1}\{\lvert\cdot\rvert^\alpha\F\{\psi\}\}$ may be reasonably interpreted as a fractional Laplacian (a differential operator); indeed,
\begin{align*}
\F^{-1}\{\lvert\cdot\rvert^\alpha\F\{\psi\}\}=(-\Delta)^{\alpha/2}\psi\qquad\text{when }\alpha\in 2\N\text{ and }\psi\in\schwartz
\end{align*}
where $\Delta$ denotes the Laplacian. Keeping in mind \eqref{eq:n_asymp}, the red highlight in
\begin{align*}
\text{D}_\ell^\alpha\{\cdot\}(t)=t^{-\alpha}\E^{-1}\{\color{red}\F^{-1}\{N_\ell^\alpha\F\{\color{black}\E\{\cdot\}\color{red}\}\}\color{black}\}(t)
\end{align*}
elucidates that $\text{D}_\ell^\alpha$ is, up to exponential substitutions, not all that dissimilar to a fractional Laplacian.

The $\text{D}_\ell^\alpha$ and $\text{\DH}_\ell^\alpha$ are in fact compositions of two different differential operators:
\begin{lem}\label{lem:t_alpha_partial_alpha}
For all $\alpha\in\N,$
\begin{align*}
\M_0^{-1}\biggl\{\frac{\Gamma(-iy)}{\Gamma(1-iy-\alpha)}\M_0\{\cdot\}\biggr\}=\iota t^\alpha\partial^{\alpha-1}t^{-1}=t^\alpha\partial^{\alpha-1}t^{-1}\iota
\end{align*}
where, as you may recall, $\iota$ embeds $\PS'$ into $\sfP'.$
\end{lem}
\begin{proof}
By Lemma \ref{lem:stirling} in conjunction with \cite[(6)]{fungkang} and \cite[(14.28)]{kolk}:
\begin{align*}
\M_0\{t^\alpha\partial^{\alpha-1}t^{-1}\psi\}&=\sum_{m=1}^\alpha s(\alpha,m)\M_0\{(t\partial)^{m-1}\psi\}=\sum_{m=1}^\alpha s(\alpha,m)\F\{\partial^{m-1}\E\{\psi\}\}\\
&=\sum_{m=1}^\alpha s(\alpha,m)(-iy)^{m-1}\M_0\{\psi\}=\frac{\Gamma(-iy)}{\Gamma(1-iy-\alpha)}\M_0\{\psi\}
\end{align*}
for all $\psi\in\PS',$ where the last equality is ensuant on \cite[(26.8.7)]{dlmf}. The desiderata are now trivial.
\end{proof}
\begin{cor}\label{cor:t_alpha_partial_alpha}
$\M_0^{-1}\biggl\{\frac{\Gamma(1-iy)}{\Gamma(1-iy-\alpha)}\M_0\{\cdot\}\biggr\}=\iota t^\alpha\partial^\alpha=t^\alpha\partial^\alpha\iota$ for all $\alpha\in\N_0$
\end{cor}
\begin{proof}
Mutatis mutandis using Corollary \ref{cor:stirling} in place of Lemma \ref{lem:stirling}. 
\end{proof}
\begin{lem}\label{lem:dh_decomp}
$\text{\DH}_\ell^\alpha=\partial^{\alpha-1}t^{-1}\text{\DH}_\ell$ for all $\alpha\in\N$
\end{lem}
\begin{proof}
As is apparent from the definitions, it suffices to prove that
\begin{align*}
\M_\alpha^{-1}\biggl\{\frac{\Gamma(-iy)}{\Gamma(1-iy-\alpha)}\M_0\{\cdot\}\biggr\}=\partial^{\alpha-1}t^{-1}\iota
\end{align*}
By Lemma \ref{lem:t_alpha_partial_alpha} it in turn suffices to show that on $\sfP'$
\begin{align}\label{eq:t_alpha_cancels}
t^{-\alpha}(t^\alpha\partial^{\alpha-1}t^{-1})=\partial^{\alpha-1}t^{-1}
\end{align}
Let $\psi\in\sfP'$ and $\omega\in\sfP.$ If $t^{-1}\partial^{\alpha-1}t^\alpha=\sum_{m=1}^\alpha\stirlingi{\alpha}{m}(\partial t)^{m-1}$ on $\sfP,$ then indeed
\begin{gather*}
\langle t^{-\alpha}(t^\alpha\partial^{\alpha-1}t^{-1})\psi\mid\omega\rangle=\langle t^\alpha\partial^{\alpha-1}t^{-1}\psi\mid t^{-\alpha}\omega\rangle=\sum_{m=1}^\alpha s(\alpha,m)\langle(t\partial)^{m-1}\psi\mid t^{-\alpha}\omega\rangle\\
=\sum_{m=1}^\alpha s(\alpha,m)(-1)^{m-1}\langle\psi\mid(\partial t)^{m-1}t^{-\alpha}\omega\rangle=(-1)^{\alpha-1}\sum_{m=1}^\alpha\stirlingi{\alpha}{m}\langle\psi\mid(\partial t)^{m-1}t^{-\alpha}\omega\rangle\\
=(-1)^{\alpha-1}\langle\psi\mid(t^{-1}\partial^{\alpha-1}t^\alpha)t^{-\alpha}\omega\rangle=(-1)^{\alpha-1}\langle t^{-1}\psi\mid\partial^{\alpha-1}\omega\rangle=\langle\partial^{\alpha-1}t^{-1}\psi\mid\omega\rangle
\end{gather*}
so we finish up by proving that $t^{-1}\partial^{\alpha-1}t^\alpha=\sum_{m=1}^\alpha\stirlingi{\alpha}{m}(\partial t)^{m-1}$ on $\sfP.$

The case $\alpha=1$ is trivial, and by induction, using that $\partial t=1+t\partial,$
\begin{gather*}
\sum_{m=1}^{\alpha+1}\stirlingi{\alpha+1}{m}(\partial t)^{m-1}=\alpha\sum_{m=1}^{\alpha+1}\stirlingi{\alpha}{m}(\partial t)^{m-1}+\sum_{m=1}^{\alpha+1}\stirlingi{\alpha}{m-1}(\partial t)^{m-1}\\
=\alpha\sum_{m=1}^\alpha\stirlingi{\alpha}{m}(\partial t)^{m-1}+\sum_{m=0}^\alpha\stirlingi{\alpha}{m}(\partial t)^m=\alpha t^{-1}\partial^{\alpha-1}t^\alpha+\sum_{m=1}^\alpha\stirlingi{\alpha}{m}(\partial t)^m\\
=\alpha t^{-1}\partial^{\alpha-1}t^\alpha+t^{-1}\partial^{\alpha-1}t^\alpha(\partial t)=\alpha t^{-1}\partial^{\alpha-1}t^\alpha+t^{-1}\partial^{\alpha-1}t^\alpha(1+t\partial)\\
=(\alpha+1)t^{-1}\partial^{\alpha-1}t^\alpha+t^{-1}\partial^{\alpha-1}t^{\alpha+1}\partial=t^{-1}\partial^{\alpha-1}(\partial t^{\alpha+1})=t^{-1}\partial^\alpha t^{\alpha+1}\tag*{\qedhere}
\end{gather*}
\end{proof}
\begin{cor}\label{cor:d_decomp}
$\text{D}_\ell^\alpha=\partial^\alpha\text{D}_\ell$ for all $\alpha\in\N_0$
\end{cor}
\begin{proof}
It suffices to prove that
\begin{align}
\M_\alpha^{-1}\biggl\{\frac{\Gamma(1-iy)}{\Gamma(1-iy-\alpha)}\M_0\{\cdot\}\biggr\}=\partial^\alpha\iota\label{eq:partial_alpha}
\end{align}
In turn, it suffices to show that $t^{-\alpha}(t^\alpha\partial^\alpha)=\partial^\alpha$ on $\sfP'$ by Corollary \ref{cor:t_alpha_partial_alpha}. But this follows immediately from mutiplying both sides of \eqref{eq:t_alpha_cancels} on the right by $t\partial$ if $\alpha\in\N,$ and is trivial if $\alpha=0.$
\end{proof}

\section{Wiener algebra \& dual Radon transform}\label{sec:main_base_case}

In this section we will essentially prove that
\begin{align*}
\F\{\ell_1\}\subset\sfL\subset\dualradon\{\ell_{1,\infty}(-\infty)\}
\end{align*}
where $\dualradon\{h\}=\textstyle\int h(w,\langle w\mid\!{}\cdot{}\rangle)\,\underline{dw}$ is the \myuline{dual Radon transform} \cite[{\S\,}3]{radonNN}.

Let us formalize what we intend to prove.
\begin{defn}
The \myuline{Wiener algebra} $\B\coloneq\{\F\{\phi\}:\phi\in\ell_1\}.$ Because $f\in\B\Rightarrow$ $\exists!\phi\in\ell_1:f=\F\{\phi\}$ we may set $\lVert f\rVert_\B\coloneq\lVert\phi\rVert_1.$
\end{defn}
The Wiener algebra is plainly a Banach space.
\begin{defn}
Let $\sfE\coloneq\{\dualradon\{h\}:h\in\ell_{1,\infty}(-\infty)\}$ with
\begin{align*}
\lVert f\rVert_\sfE\coloneq\inf\Bigl\{\lVert h\rVert_{1,\infty}:f=\dualradon\{h\}\Bigr\}
\end{align*}
\end{defn}
Note that $\lVert\cdot\rVert_\sfE$ is a seminorm because $f_+,f_-\in\sfE\Rightarrow$
\begin{align*}
\lVert f_++f_-\rVert_\sfE&=\inf\Bigl\{\lVert h\rVert_{1,\infty}:f_++f_-=\dualradon\{h\}\Bigr\}\\
&\leq\inf\Bigl\{\lVert h^++h^-\rVert_{1,\infty}:f_\pm=\dualradon\{h^\mp\}\Bigr\}\\
&\leq\inf\Bigl\{\lVert h^+\rVert_{1,\infty}+\lVert h^-\rVert_{1,\infty}:f_\pm=\dualradon\{h^\mp\}\Bigr\}\\
&=\lVert f_+\rVert_\sfE+\lVert f_-\rVert_\sfE
\end{align*}
The main result of this section:
\begin{theorem}\label{thm:main_base_case}
$\B\subset\sfL\subset\sfE$ with $\lVert f\rVert_\sfE\leq\lVert f\rVert_\sfL$ if $f\in\sfL$ and $\lVert f\rVert_\sfL\leq\lVert f\rVert_\B$ if $f\in\B$
\end{theorem}

\subsection{Proof of Theorem \ref{thm:main_base_case}}

We first characterize the $\text{D}_\ell$ and the $\text{\DH}_\ell.$ We remind the reader that $\regPS'=\PS'\cap\ell_1^\text{loc}(0,\infty).$
\begin{lem}\label{lem:d_inverted}
Let $f\in\PS'$ and $G\in\regPS'.$ Then $G=\text{D}_\ell\{f\}\Leftrightarrow$
\begin{align}
f[t]=2\lvert\sphere^{n-2}\rvert\int_0^1 G(tv)C_\ell^{(n-2)/2}(v)(1-v^2)^{(n-3)/2}\,dv\label{eq:d_inverted}
\end{align}
so the distribution $f[t]$ is a function $f(t),$ whereas $G=\text{\DH}_\ell\{f\}\Leftrightarrow$
\begin{align}
f[t]=2\lvert\sphere^{n-2}\rvert\int_0^1 G(tv)\tfrac{1}{n-1}C_{\ell-1}^{n/2}(v)(1-v^2)^{(n-1)/2}\,\frac{dv}{v}\label{eq:dh_inverted}
\end{align}
\end{lem}
\begin{rem}
The integrals are well-defined because $G$ is integrable on $[0,t]$ and $C_{\ell-1}^{n/2}(v)/v$ is an even polynomial whenever $\ell\in L.$
\end{rem}
\begin{proof}
Since $1/N_\ell$ is smooth (and well-defined because $N_\ell$ has no zeros), arguing like we did in Appendix \ref{sec:NOM} yields that $1/N_\ell\in O_C\subset O_M,$ whence
\begin{align*}
\M_0\{G\}=N_\ell\M_0\{f\}\Leftrightarrow(1/N_\ell)\M_0\{G\}=\M_0\{f\}
\end{align*}
\fbox{Proof that $G=\text{D}_\ell\{f\}\Leftrightarrow\eqref{eq:d_inverted}$}

Since $\M_0$ is an isomorphism, $G=\text{D}_\ell\{f\}\Leftrightarrow\M_0\{G\}=N_\ell\M_0\{f\}\Leftrightarrow$ $(1/N_\ell)\M_0\{G\}=\M_0\{f\}.$ Now, since \cite[9.20]{oberhettinger}
\begin{align*}
1/N_\ell(y)=\frac{2\displaystyle\pi^{(n-1)/2}\Gamma(\tfrac{2-iy}{2})\Gamma(\tfrac{1-iy}{2})}{\Gamma(\tfrac{n+\ell-iy}{2})\Gamma(\tfrac{2-\ell-iy}{2})}&=\\
2\lvert\sphere^{n-2}\rvert\M\Bigl\{v\mapsto C_\ell^{(n-2)/2}(v)(1-v^2)^{(n-3)/2}[v<1]\Bigr\}(1-iy)&=\\
2\lvert\sphere^{n-2}\rvert\int_0^1 v^{1-iy}C_\ell^{(n-2)/2}(v)(1-v^2)^{(n-3)/2}\,\frac{dv}{v}&=\\
2\lvert\sphere^{n-2}\rvert\int_1^\infty t^{iy-1}C_\ell^{(n-2)/2}(1/t)(1-(1/t)^2)^{(n-3)/2}\,\frac{dt}{t}&=\\
2\lvert\sphere^{n-2}\rvert\M\Bigl\{t\mapsto(1/t)C_\ell^{(n-2)/2}(1/t)(1-(1/t)^2)^{(n-3)/2}[t>1]\Bigr\}(iy)&=\\
\M_0\Bigl\{t\mapsto 2\lvert\sphere^{n-2}\rvert(1/t)C_\ell^{(n-2)/2}(1/t)(1-(1/t)^2)^{(n-3)/2}[t>1]\Bigr\}(y)
\end{align*}
it follows from \cite[(VII, 8; 5)]{schwarz} that $G=\text{D}_\ell\{f\}\Leftrightarrow\F\{\E\{f\}\}=\M_0\{f\}=$
\begin{gather*}
\M_0\{G\}\M_0\Bigl\{t\mapsto 2\lvert\sphere^{n-2}\rvert(1/t)C_\ell^{(n-2)/2}(1/t)(1-(1/t)^2)^{(n-3)/2}[t>1]\Bigr\}\\
=\F\{\E\{G\}\}\F\Bigl\{\E\Bigl\{t\mapsto 2\lvert\sphere^{n-2}\rvert(1/t)C_\ell^{(n-2)/2}(1/t)(1-(1/t)^2)^{(n-3)/2}[t>1]\Bigr\}\Bigr\}\\
=\F\Bigl\{\E\{G\}*\E\Bigl\{t\mapsto 2\lvert\sphere^{n-2}\rvert(1/t)C_\ell^{(n-2)/2}(1/t)(1-(1/t)^2)^{(n-3)/2}[t>1]\Bigr\}\Bigr\}
\end{gather*}
which  is in turn tantamount to \cite[(3)]{fungkang}
\begin{align}\label{eq:f_l_as_convolution}
f(e^u)&=2\lvert\sphere^{n-2}\rvert\int_1^\infty G(e^{u-z})e^{-z}C_\ell^{(n-2)/2}(e^{-z})(1-e^{-2z})^{(n-3)/2}\,dz
\end{align}
because $\F\!:\schwartz'\to\schwartz'$ is an isomorphism. Letting $t=e^u$ and $v=e^{-z}$ finally yields that $G=\text{D}_\ell\{f\}\Leftrightarrow$
\begin{align*}
f(t)=2\lvert\sphere^{n-2}\rvert\int_0^1 G(tv)C_\ell^{(n-2)/2}(v)(1-v^2)^{(n-3)/2}\,dv
\end{align*}
\fbox{Proof that $G=\text{\DH}_\ell\{f\}\Leftrightarrow\eqref{eq:dh_inverted}$}

For $\ell\in L,$ \cite[9.20]{oberhettinger} yields that
\begin{gather*}
\M\Bigl\{v\mapsto 2\lvert\sphere^{n-2}\rvert\tfrac{1}{n-1}C_{\ell-1}^{n/2}(v)(1-v^2)^{(n-1)/2}[v<1](1/v)\Bigr\}(1+s)=\\
\M\Bigl\{v\mapsto 2\lvert\sphere^{n-2}\rvert\tfrac{1}{n-1}C_{\ell-1}^{n/2}(v)(1-v^2)^{(n-1)/2}[v<1]\Bigr\}(s)=\frac{\displaystyle\pi^{(n-1)/2}\Gamma(\tfrac{s}{2})\Gamma(\tfrac{1+s}{2})}{\Gamma(\tfrac{n+\ell+s}{2})\Gamma(\tfrac{2-\ell+s}{2})}
\end{gather*}
for $\Re(s)>0.$ However, since both sides of the equality
\begin{align*}
\M\Bigl\{v\mapsto 2\lvert\sphere^{n-2}\rvert\tfrac{1}{n-1}C_{\ell-1}^{n/2}(v)(1-v^2)^{(n-1)/2}[v<1]\Bigr\}(s)=\frac{\displaystyle\pi^{(n-1)/2}\Gamma(\tfrac{s}{2})\Gamma(\tfrac{1+s}{2})}{\Gamma(\tfrac{n+\ell+s}{2})\Gamma(\tfrac{2-\ell+s}{2})}
\end{align*}
are holomorphic for $\Re(s)>-2$ \cite{mattner}, it follows that
\begin{gather*}
\M\Bigl\{v\mapsto 2\lvert\sphere^{n-2}\rvert\tfrac{1}{n-1}C_{\ell-1}^{n/2}(v)/v\cdot(1-v^2)^{(n-1)/2}[v<1]\Bigr\}(1-iy)\\
=\frac{\displaystyle\pi^{(n-1)/2}\Gamma(\tfrac{-iy}{2})\Gamma(\tfrac{1-iy}{2})}{\Gamma(\tfrac{n+\ell-iy}{2})\Gamma(\tfrac{2-\ell-iy}{2})}
\end{gather*}
because analytic continuations are unique. The rest of the proof now follows mutatis mutandis from the proof of $G=\text{D}_\ell\{f\}\Leftrightarrow\eqref{eq:d_inverted}.$
\end{proof}
We will now establish the effect of the $\text{D}_\ell$ and the $\text{\DH}_\ell$ on $f$ as in Definition \ref{def:gamma_d}.
\begin{lem}\label{lem:dh_inverted}
If $f,F=-\int^\to f(v)\,\frac{dv}{v}\in\PS',$ then $\regPS'\ni G=\text{\DH}_\ell\{F\}\Rightarrow$
\begin{align*}
f(t)=2\lvert\sphere^{n-2}\rvert\int_0^1 G(tv)C_\ell^{(n-2)/2}(v)(1-v^2)^{(n-3)/2}\,dv\label{eq:dh_invertedF}\tag{\ref*{eq:d_inverted}'}
\end{align*}
\end{lem}
\begin{proof}
Like in the proof of Lemma \ref{lem:d_inverted}, $1/\text{``}\!N_\ell^0\text{''}\in O_M,$ where
\begin{align*}
1/\text{``}\!N_\ell^0\text{''}:\R\ni y\mapsto\frac{2\displaystyle\pi^{(n-1)/2}\Gamma(\tfrac{2-iy}{2})\Gamma(\tfrac{1-iy}{2})}{\Gamma(\tfrac{n+\ell-iy}{2})\Gamma(\tfrac{2-\ell-iy}{2})}
\end{align*}
Multiplying both sides of $\M_0\{G\}=N_\ell\M_0\{F\}$ by $1/\text{``}\!N_\ell^0\text{''}$ yields that
\begin{align*}
(1/\text{``}\!N_\ell^0\text{''})\M_0\{G\}=(N_\ell/\text{``}\!N_\ell^0\text{''})\M_0\{F\}
\end{align*}
Now, $(N_\ell/\text{``}\!N_\ell^0\text{''})(y)=\frac{2\Gamma(\tfrac{2-iy}{2})}{\Gamma(\tfrac{-iy}{2})}=\frac{2\Gamma(1+\tfrac{-iy}{2})}{\Gamma(\tfrac{-iy}{2})}=-iy,$ so
\begin{gather}\label{eq:F->f}
\begin{gathered}
\frac{2\displaystyle\pi^{(n-1)/2}\Gamma(\tfrac{2-iy}{2})\Gamma(\tfrac{1-iy}{2})}{\Gamma(\tfrac{n+\ell-iy}{2})\Gamma(\tfrac{2-\ell-iy}{2})}\M_0\{G\}[y]=\\
-iy\M_0\{F\}[y]=\F\{\partial\E\{F\}\}[y]=\M_0\{t\mapsto t\partial F(t)\}[y]=\M_0\{f\}[y]
\end{gathered}
\end{gather}
by \cite[(14.28)]{kolk} and \cite[(6)]{fungkang}. The rest of the proof is as the proof of Lemma \ref{lem:d_inverted}.
\end{proof}
\begin{rem}
Actually, even when $G=\text{\DH}_\ell\{F\},$ \eqref{eq:dh_invertedF} uniquely determines $G$ \myuline{up to an} \myuline{additive constant}. The only difference is that ``\!$N_\ell^0$'' has a pole at 0, so the claim follows from the linearity of \eqref{eq:dh_invertedF} as a function $G\mapsto f$ together with the fact that $f=0\Rightarrow(1/\text{``}\!N_\ell^0\text{''})\M_0\{G\}=\M_0\{f\}=0\Rightarrow\M_0\{G\}$ is a multiple of $\delta_0,$ implying that $G$ is a constant function.
\end{rem}
\begin{cor}\label{cor:G_uniquely_determined}
Let $f$ and $\{G_\ell\}_{\ell=0}^\infty$ be as in Definition \ref{def:gamma_d} (so $G_\ell\in\ell_{1,\infty}(0)$ for every $\ell$ and $\{f_\ell(\theta,\cdot)\}_{\ell=0}^\infty\in\sfF$ for all $\theta).$ Then
\begin{align}
f_\ell(\theta,t)=2\lvert\sphere^{n-2}\rvert\int_0^1 G_\ell(\theta,tv)C_\ell^{(n-2)/2}(v)(1-v^2)^{(n-3)/2}\,dv\label{eq:f_in_terms_of_g}\tag{\ref*{eq:d_inverted}''}
\end{align}
If $\ell\notin L,$ \eqref{eq:f_in_terms_of_g} uniquely determines $G_\ell.$ If $\ell\in L,$ then $G_\ell$ is instead uniquely determined by
\begin{align}
F_\ell(\theta,t)=2\lvert\sphere^{n-2}\rvert\int_0^1 G_\ell(\theta,tv)\tfrac{1}{n-1}C_{\ell-1}^{n/2}(v)(1-v^2)^{(n-1)/2}\,\frac{dv}{v}\label{eq:F}\tag{\ref*{eq:dh_inverted}'}
\end{align}
where $F_\ell$ was defined in Definition \ref{def:F}.
\end{cor}
\begin{proof}
If $G_\ell\in\ell_{1,\infty}(0),$ then $G_\ell(\theta,\cdot)\in\ell_\infty(0,\infty)$ for $\theta\in\Omega,$ where $\Omega\subset\sphere^{n-1}$ is conull. Applying Lemmas \ref{lem:d_inverted} and \ref{lem:dh_inverted} to $f_\ell(\theta,\cdot)$ and $G_\ell(\theta,\cdot)$ with $\theta\in\Omega$ fixed arbitrarily yields the desiderata for almost every $\theta.$
\end{proof}

\fbox{Proof that $\lVert f\rVert_\sfL\leq\lVert f\rVert_\B$ if $f\in\B$}

Let $f\in\B.$ We will step by step derive a closed form for $h$ in terms of $\phi.$ With the help of Corollary \ref{cor:G_uniquely_determined}, we can first derive a closed form for $G_\ell$ in terms of $\phi.$
\begin{lem}\label{lem:g_closed_form}
$\B\subset\sfF$ and, for all $\ell\in\N_0,$
\begin{align*}
G_\ell(\theta,t)=\int\phi(u)Z_\ell(\theta,\hat{u})\cis_\ell(t\lvert u\rvert)\,du\in\ell_{1,\infty}(\underline{d\theta},dt/t)
\end{align*}
\end{lem}
\begin{proof}
Let $G_\ell$ be as given in the statement of the Lemma. That $\lVert G_\ell\rVert_{1,\infty}\lesssim_{n,\ell}\lVert\phi\rVert_1<\infty$ follows from a Lemma \ref{lem:z_bound}-like argument.

If $f\in\B,$ then $\exists!\phi\in\ell_1:f=\F\{\phi\}.$ Let us compute $f_\ell.$
\begin{align}
f_\ell(\theta,t)&=\int f(t\alpha)Z_\ell(\theta,\alpha)\,\underline{d\alpha}=\iint\phi(u)\exp(it\langle u\mid\alpha\rangle)Z_\ell(\theta,\alpha)\,du\,\underline{d\alpha}\notag\\
&\overset{\text{F}}{=}\int\phi(u)\int\exp(it\lvert u\rvert\langle\hat{u}\mid\alpha\rangle)Z_\ell(\theta,\alpha)\,\underline{d\alpha}\,du\label{eq:zonal_fourier}
\end{align}
Consequently, \eqref{eq:z_propto_c} and Proposition \ref{prop:funk_hecke_gegenbauer} yield that
\begin{align*}
\int\exp(it\lvert u\rvert\langle\hat{u}\mid\alpha\rangle)Z_\ell(\theta,\alpha)\,\underline{d\alpha}&=\\
Z_\ell(\theta,\hat{u})\lvert\sphere^{n-2}\rvert\int_{-1}^{+1}\exp(it\lvert u\rvert v)C^{(n-2)/2}_\ell(v)(1-v^2)^{(n-3)/2}\,dv&=\\
Z_\ell(\theta,\hat{u})(2\pi)^{n/2}i^\ell J_{\ell+(n-2)/2}(t\lvert u\rvert)/(t\lvert u\rvert)^{(n-2)/2}
\end{align*}
by \cite[(3) on pg.\ 50]{watson}. Plugging back into \eqref{eq:zonal_fourier} we see that
\begin{align*}
f_\ell(\theta,t)&=(2\pi)^{n/2}i^\ell\int\phi(u)Z_\ell(\theta,\hat{u})J_{\ell+(n-2)/2}(t\lvert u\rvert)/(t\lvert u\rvert)^{(n-2)/2}\,du
\end{align*}
To show that \eqref{eq:f_in_terms_of_g} holds with the given $G_\ell,$ we have to show that
\begin{gather*}
2\lvert\sphere^{n-2}\rvert\int_0^1\int\phi(u)Z_\ell(\theta,\hat{u})\cis_\ell(tv\lvert u\rvert)C_\ell^{(n-2)/2}(v)(1-v^2)^{(n-3)/2}\,du\,dv=\\
(2\pi)^{n/2}i^\ell\int\phi(u)Z_\ell(\theta,\hat{u})J_{\ell+(n-2)/2}(t\lvert u\rvert)/(t\lvert u\rvert)^{(n-2)/2}\,du
\end{gather*}
Indeed, \cite[pp.\ 38 \& 94]{erdelyi}
\begin{align*}
2\lvert\sphere^{n-2}\rvert\int_0^1\int\phi(u)Z_\ell(\theta,\hat{u})\cis_\ell(tv\lvert u\rvert)C_\ell^{(n-2)/2}(v)(1-v^2)^{(n-3)/2}\,du\,dv&\overset{\text{F}}{=}\\
2\lvert\sphere^{n-2}\rvert\int\phi(u)Z_\ell(\theta,\hat{u})\int_0^1\cis_\ell(t\lvert u\rvert v)C_\ell^{(n-2)/2}(v)(1-v^2)^{(n-3)/2}\,dv\,du&=\\
2\lvert\sphere^{n-2}\rvert\int\phi(u)Z_\ell(\theta,\hat{u})\Bigl(i^\ell 2^{(2-n)/2}\pi\frac{\Gamma(n-2)}{\Gamma(\tfrac{n-2}{2})}J_{\ell+(n-2)/2}(t\lvert u\rvert)/(t\lvert u\rvert)^{(n-2)/2}\Bigr)\,du&=\\
(2\pi)^{n/2}i^\ell\int\phi(u)Z_\ell(\theta,\hat{u})J_{\ell+(n-2)/2}(t\lvert u\rvert)/(t\lvert u\rvert)^{(n-2)/2}\,du
\end{align*}
by the Legendre duplication formula. So \eqref{eq:f_in_terms_of_g} holds for all $\ell.$

Seeing as \eqref{eq:f_in_terms_of_g} uniquely determines $G_\ell$ when $\ell\notin L,$ we are done with that case if $f_\ell(\theta,\cdot)\in\PS';$ indeed, $G_\ell(\theta,t)\asymplim 1$ if $t\downarrow 0$ and vanishes at infinity if $t\uparrow\infty$ by the Riemann-Lebesgue lemma, so the bounded convergence theorem applied to \eqref{eq:f_in_terms_of_g} yields that $f_\ell(\theta,\cdot)\in\UCB(0,\infty)\subset\PS'.$

\fbox{$\ell\in L$}

Courtesy of \cite[pg.\ 737]{prudnikov}:
\begin{align*}
\partial\Bigl[-\tfrac{1}{n-1}C_{\ell-1}^{n/2}(t)(1-t^2)^{(n-1)/2}\Bigr]&=\\
(1-t^2)^{(n-3)/2}\Bigl[tC_{\ell-1}^{n/2}(t)-\tfrac{(\ell-1)(n+\ell-1)}{(n-1)(n+1)}(1-t^2)C_{\ell-2}^{(n+2)/2}(t)\Bigr]&=\\
(1-t^2)^{(n-3)/2}\Bigl[\tfrac{n+\ell-1}{n-1}C_\ell^{n/2}(t)-\tfrac{\ell}{n-1}tC_{\ell-1}^{n/2}(t)\Bigr]&=\\
C_\ell^{(n-2)/2}(t)(1-t^2)^{(n-3)/2}
\end{align*}
wherefrom it follows that $-\tfrac{1}{n-1}C_{\ell-1}^{n/2}(t)(1-t^2)^{(n-1)/2}[t<1]$ is an antiderivative of $C_\ell^{(n-2)/2}(t)(1-t^2)^{(n-3)/2}[t<1].$ Consequently, if $m>t>0,$

\begin{align*}
\int_t^m f_\ell(\theta,v)\,\frac{dv}{v}&=\int_t^m 2\lvert\sphere^{n-2}\rvert\int_0^1 G_\ell(\theta,vr)C_\ell^{(n-2)/2}(r)(1-r^2)^{(n-3)/2}\,dr\,\frac{dv}{v}\\
&\overset{\text{F}}{=}2\lvert\sphere^{n-2}\rvert\int_0^1 C_\ell^{(n-2)/2}(r)(1-r^2)^{(n-3)/2}\int_t^m G_\ell(\theta,vr)\,\frac{dv}{v}\,dr\\
&=2\lvert\sphere^{n-2}\rvert\int_0^1 C_\ell^{(n-2)/2}(r)(1-r^2)^{(n-3)/2}\int_{tr}^{\mathclap{mr}}G_\ell(\theta,p)\,\frac{dp}{p}\,dr\\
&=2\lvert\sphere^{n-2}\rvert\int_0^1 C_\ell^{(n-2)/2}(r)(1-r^2)^{(n-3)/2}\int_r^{\mathclap{mr/t}} G_\ell(\theta,tv)\,\frac{dv}{v}\,dr\\
&=2\lvert\sphere^{n-2}\rvert\int_0^\infty C_\ell^{(n-2)/2}(r)(1-r^2)^{(n-3)/2}[r<1]\int_r^{\mathclap{mr/t}} G_\ell(\theta,tv)\,\frac{dv}{v}\,dr\\
&\overset{\text{F}}{=}2\lvert\sphere^{n-2}\rvert\int_0^\infty G_\ell(\theta,tv)\int_{\mathclap{tv/m}}^v C_\ell^{(n-2)/2}(r)(1-r^2)^{(n-3)/2}[r<1]\,dr\,\frac{dv}{v}\\
&=2\lvert\sphere^{n-2}\rvert\int_0^\infty G_\ell(\theta,tv)\tfrac{1}{n-1}C_{\ell-1}^{n/2}(tv/m)(1-(tv/m)^2)^{(n-1)/2}[tv/m<1]\,\frac{dv}{v}\\
&-2\lvert\sphere^{n-2}\rvert\int_0^\infty G_\ell(\theta,tv)\tfrac{1}{n-1}C_{\ell-1}^{n/2}(v)(1-v^2)^{(n-1)/2}[v<1]\,\frac{dv}{v}\\
&=2\lvert\sphere^{n-2}\rvert\int_0^1 G_\ell(\theta,mr)\tfrac{1}{n-1}C_{\ell-1}^{n/2}(r)(1-r^2)^{(n-1)/2}\,\frac{dr}{r}\\
&-2\lvert\sphere^{n-2}\rvert\int_0^1 G_\ell(\theta,tv)\tfrac{1}{n-1}C_{\ell-1}^{n/2}(v)(1-v^2)^{(n-1)/2}\,\frac{dv}{v}
\end{align*}
Since \eqref{eq:F} uniquely determines $G_\ell,$ we are done if we can show that a)
\begin{align*}
F_\ell(\theta,t)=-\int_t^\to f_\ell(\theta,v)\,\frac{dv}{v}=2\lvert\sphere^{n-2}\rvert\int_0^1 G_\ell(\theta,tv)\tfrac{1}{n-1}C_{\ell-1}^{n/2}(v)(1-v^2)^{(n-1)/2}\,\frac{dv}{v}
\end{align*}
and b) $F_\ell(\theta,\cdot)\in\PS'.$ The former follows precisely when
\begin{align}\label{eq:dct_riemann_lebesgue}
\lim_{m\to\infty}2\lvert\sphere^{n-2}\rvert\int_0^1 G_\ell(\theta,mr)\tfrac{1}{n-1}C_{\ell-1}^{n/2}(r)(1-r^2)^{(n-1)/2}\,\frac{dr}{r}=0
\end{align}
That \eqref{eq:dct_riemann_lebesgue} is indeed true follows from the bounded convergence theorem and the Riemann-Lebesgue lemma (to conclude that $\lim_{m\to\infty}G_\ell(\theta,mr)=0).$ Lastly, that $F_\ell(\theta,\cdot)\in\PS'$ follows from applying the bounded convergence theorem to \eqref{eq:F}.
\end{proof}
Following Definition \ref{def:gamma_d}, Lemma \ref{lem:g_closed_form} yields that
\begin{align*}
h_\ell(\theta,t)=\int\phi(u)Z_\ell(\theta,\hat{u})\cis_\ell(t\lvert u\rvert)\,du\in\ell_{1,\infty}(\underline{d\theta},dt)
\end{align*}
Since $u\mapsto Z_\ell(\theta,\hat{u})\propto C_\ell(\langle\theta\mid\hat{u}\rangle)$ has the same parity as $\ell,$ we may write
\begin{align*}
h_\ell(\theta,t)=\int\phi_e(u)Z_\ell(\theta,\hat{u})\cos(t\lvert u\rvert)\,du
\end{align*}
for even $\ell,$ while for odd $\ell$ the corresponding formula reads
\begin{align*}
h_\ell(\theta,t)=i\int\phi_o(u)Z_\ell(\theta,\hat{u})\sin(t\lvert u\rvert)\,du
\end{align*}
where $\phi_e$ and $\phi_o$ are the even and odd parts of $\phi$ resp. Doing a more precise Lemma \ref{lem:z_bound}-like argument reveals that $\lVert h_\ell\rVert_{1,\infty}\preccurlyeq_n\ell^{\mkern1mu n-2}\lVert\phi\rVert_1$ as $\ell\to\infty,$ so
\begin{align*}
\sum_{\ell=0}^\infty R^\ell\lVert h_\ell\rVert_{1,\infty}<\infty
\end{align*}
for $R\in(0,1).$ It thus follows that
\begin{align*}
h_R(\theta,t)&=\sum_{\ell=0}^\infty R^\ell h_\ell(\theta,t)=\sum_{\ell=0}^\infty R^\ell\int\phi(u)Z_\ell(\theta,\hat{u})\cis_\ell(t\lvert u\rvert)\,du\\
&=\sum_{2\mid\ell}R^\ell\int\phi_e(u)Z_\ell(\theta,\hat{u})\cos(t\lvert u\rvert)\,du+i\sum_{2\nmid\ell}R^\ell\int\phi_o(u)Z_\ell(\theta,\hat{u})\sin(t\lvert u\rvert)\,du\\
&=\sum_{\ell=0}^\infty R^\ell\int\phi_e(u)Z_\ell(\theta,\hat{u})\cos(t\lvert u\rvert)\,du+i\sum_{\ell=0}^\infty R^\ell\int\phi_o(u)Z_\ell(\theta,\hat{u})\sin(t\lvert u\rvert)\,du\\
&\overset{\text{F}}{=}\int\phi_e(u)p(\hat{u},R\theta)\cos(t\lvert u\rvert)\,du+i\int\phi_o(u)p(\hat{u},R\theta)\sin(t\lvert u\rvert)\,du
\end{align*}
\begin{lem}\label{lem:h}
$h_R\to h$ in $\ell_{1,\infty}(-\infty)$ where
\begin{align}\label{eq:h}
h(\theta,t)=\int_0^\infty\rho^{n-1}\phi_e(\theta\rho)\cos(t\rho)\,d\rho+i\int_0^\infty\rho^{n-1}\phi_o(\theta\rho)\sin(t\rho)\,d\rho
\end{align}
\end{lem}
\begin{proof}
Let $\Phi:\sphere^{n-1}\times\R\ni(\theta,t)\mapsto\textstyle\int_0^\infty\displaystyle\rho^{n-1}\phi_e(\theta\rho)\cos(t\rho)\,d\rho.$ Note that
\begin{align*}
\lVert\Phi\rVert_{1,\infty}\leq\int\int_0^\infty\rho^{n-1}\lvert\phi_e(\theta\rho)\rvert\,d\rho\,\underline{d\theta}=\lVert\phi_e\rVert_1<\infty
\end{align*}
so $\Phi\in\ell_{1,\infty}(-\infty)$ and $\Phi(\theta,\cdot)$ is well-defined for $\theta\in\Omega,$ where $\Omega$ is conull. We therefore confine $\alpha,\theta\in\Omega$ throughout.
 
We show that
\begin{align*}
\int\phi_e(u)p(\hat{u},R\theta)\cos(t\lvert u\rvert)\,du&=\int p(\alpha,R\theta)\Phi(\alpha,t)\,\underline{d\alpha}
\end{align*}
converges to $\Phi(\theta,t)$ in $\ell_{1,\infty}(\underline{d\theta},dt)$ if $R\uparrow1.$ That $\textstyle\int\displaystyle\phi_o(u)p(\hat{u},R\theta)\sin(t\lvert u\rvert)\,du\to$ $\textstyle\int_0^\infty\displaystyle\rho^{n-1}\phi_o(\theta,\rho)\sin(t\rho)\,d\rho$ in $\ell_{1,\infty}(\underline{d\theta},dt)$ if $R\uparrow1$ may then be proved similarly.

Since $\Phi(\theta,\cdot)$ is continuous by dominated convergence, it follows that
\begin{align*}
\lVert\Phi(\alpha,\cdot)-\Phi(\theta,\cdot)\rVert_\infty=\sup_{t\in\R}\,\lvert\Phi(\alpha,t)-\Phi(\theta,t)\rvert=\sup_{t\in\mathds{Q}}\,\lvert\Phi(\alpha,t)-\Phi(\theta,t)\rvert
\end{align*}
so $(\theta,\alpha)\mapsto\lVert\Phi(\alpha,\cdot)-\Phi(\theta,\cdot)\rVert_\infty,$ being a countable supremum of measurable functions, is measurable.

By Proposition \ref{prop:ii_1_9}(b), $\Phi(\theta,t)=\textstyle\int\displaystyle p(\alpha,R\theta)\Phi(\theta,t)\,\underline{d\alpha},$ so by Proposition \ref{prop:ii_1_9}(a) it suffices to show that
\begin{gather*}
\lVert\int p(\alpha,R\theta)\Phi(\alpha,\cdot)\,\underline{d\alpha}-\Phi(\theta,\cdot)\mid\ell_{1,\infty}(\underline{d\theta},\R)\rVert\\
\leq\iint p(\alpha,R\theta)\lVert\Phi(\alpha,\cdot)-\Phi(\theta,\cdot)\rVert_\infty\,\underline{d\alpha}\,\underline{d\theta}
\end{gather*}
can be made arbitrarily small by choosing $R$ arbitrarily close to 1.

Seeing as integrable Lipschitz continuous functions are dense in $\ell_1,$ it follows that $\forall\epsilon>0$ there exists a Lipschitz continuous $\psi_\epsilon\in\ell_1$ so that $\lVert\phi_e-\psi_\epsilon\rVert_1<\epsilon.$ Letting $\Psi:\Omega\times\R\ni(\theta,t)\mapsto\textstyle\int_0^\infty\displaystyle\rho^{n-1}\psi_\epsilon(\theta\rho)\cos(t\rho)\,d\rho,$ we have
\begin{align*}
\lVert\Phi-\Psi\rVert_{1,\infty}\leq\lVert\phi_e-\psi_\epsilon\rVert_1<\epsilon
\end{align*}
By the same reasoning as earlier, $\alpha\mapsto\lVert\Phi(\alpha,\cdot)-\Psi(\alpha,\cdot)\rVert_\infty$ and $(\theta,\alpha)\mapsto$ $\lVert\Psi(\alpha,\cdot)-\Psi(\theta,\cdot)\rVert_\infty$ are measurable. Ergo, $p(\alpha,R\theta)=p(\theta,R\alpha)\Rightarrow$
\begin{gather*}
\iint p(\alpha,R\theta)\lVert\Phi(\alpha,\cdot)-\Phi(\theta,\cdot)\rVert_\infty\,\underline{d\alpha}\,\underline{d\theta}\leq\iint p(\alpha,R\theta)\lVert\Phi(\alpha,\cdot)-\Psi(\alpha,\cdot)\rVert_\infty\,\underline{d\alpha}\,\underline{d\theta}\\
+\iint p(\alpha,R\theta)\lVert\Psi(\alpha,\cdot)-\Psi(\theta,\cdot)\rVert_\infty\,\underline{d\alpha}\,\underline{d\theta}+\iint p(\alpha,R\theta)\lVert\Psi(\theta,\cdot)-\Phi(\theta,\cdot)\rVert_\infty\,\underline{d\alpha}\,\underline{d\theta}\\
\leq2\lVert\Phi-\Psi\rVert_{1,\infty}+\iint p(\alpha,R\theta)\lVert\Psi(\alpha,\cdot)-\Psi(\theta,\cdot)\rVert_\infty\,\underline{d\alpha}\,\underline{d\theta}\\
<2\epsilon+\iint p(\alpha,R\theta)\lVert\Psi(\alpha,\cdot)-\Psi(\theta,\cdot)\rVert_\infty\,\underline{d\alpha}\,\underline{d\theta}
\end{gather*}
Now let $\Psi_a:\Omega\times\R\ni(\theta,t)\mapsto\textstyle\int_0^a\displaystyle\rho^{n-1}\psi_\epsilon(\theta\rho)\cos(t\rho)\,d\rho$ with $a>0.$ Then
\begin{align*}
\lVert\Psi-\Psi_a\rVert_{1,\infty}\leq\int_{\mathclap{\lvert u\rvert>a}}\,\lvert\psi_\epsilon(u)\rvert\,du\to0
\end{align*}
as $a\to\infty$ by dominated convergence. We may accordingly pick $a$ so large that $\lVert\Psi-\Psi_a\rVert_{1,\infty}<\epsilon,$ wherefrom we may again deduce that
\begin{gather*}
\iint p(\alpha,R\theta)\lVert\Psi(\alpha,\cdot)-\Psi(\theta,\cdot)\rVert_\infty\,\underline{d\alpha}\,\underline{d\theta}<\\
2\epsilon+\iint p(\alpha,R\theta)\lVert\Psi_a(\alpha,\cdot)-\Psi_a(\theta,\cdot)\rVert_\infty\,\underline{d\alpha}\,\underline{d\theta}
\end{gather*}
We are now in a position to adapt the proof of \cite[Thm II.1.10]{steinweiss}; i.e., writing the addend in the above display as $\textstyle\iint=\iint_{\lvert\theta-\alpha\rvert>\delta}+\iint_{\lvert\theta-\alpha\rvert\leq\delta}$ with $\delta>0,$ the $\textstyle\iint_{\lvert\theta-\alpha\rvert\leq\delta}$ term may be bounded as follows:
\begin{gather*}
\iint_{\mathclap{\lvert\theta-\alpha\rvert\leq\delta}}p(\alpha,R\theta)\lVert\Psi_a(\alpha,\cdot)-\Psi_a(\theta,\cdot)\rVert_\infty\,\underline{d\alpha}\,\underline{d\theta}\leq\lvert\sphere^{n-1}\rvert\sup_{\theta\in\Omega}\int_{\mathclap{\lvert\theta-\alpha\rvert\leq\delta}}\,\lVert\Psi_a(\alpha,\cdot)-\Psi_a(\theta,\cdot)\rVert_\infty\,\underline{d\alpha}\\
\lesssim_n\sup_{\theta\in\Omega}\int\nolimits_{\lvert\theta-\alpha\rvert\leq\delta}\int_0^a\rho^{n-1}\lvert\psi_\epsilon(\alpha\rho)-\psi_\epsilon(\theta\rho)\rvert\,d\rho\,\underline{d\alpha}\leq\sup_{\theta\in\Omega}\int\nolimits_{\lvert\theta-\alpha\rvert\leq\delta}\int_0^a\rho^{n-1}[\delta a\mathrm{Lip}(\psi_\epsilon)]\,d\rho\,\underline{d\alpha}\\
\leq\delta a\mathrm{Lip}(\psi_\epsilon)\sup_{\theta\in\Omega}\int\int_0^a\rho^{n-1}\,d\rho\,\underline{d\alpha}=\delta a\mathrm{Lip}(\psi_\epsilon)\lvert K(a)\rvert<\epsilon
\end{gather*}
if we pick $\delta$ small enough. Lastly, using $p(\alpha,R\theta)=p(\theta,R\alpha)$ again,
\begin{gather*}
\iint_{\mathclap{\lvert\theta-\alpha\rvert>\delta}}p(\alpha,R\theta)\lVert\Psi_a(\alpha,\cdot)-\Psi_a(\theta,\cdot)\rVert_\infty\,\underline{d\alpha}\,\underline{d\theta}\leq\iint_{\mathclap{\lvert\theta-\alpha\rvert>\delta}}p(\alpha,R\theta)\lVert\Psi_a(\alpha,\cdot)\rVert_\infty\,\underline{d\alpha}\,\underline{d\theta}\\
+\iint_{\mathclap{\lvert\theta-\alpha\rvert>\delta}}p(\alpha,R\theta)\lVert\Psi_a(\theta,\cdot)\rVert_\infty\,\underline{d\alpha}\,\underline{d\theta}\overset{\text{F}}{=}\int\lVert\Psi_a(\alpha,\cdot)\rVert_\infty\int_{\mathclap{\lvert\theta-\alpha\rvert>\delta}}p(\theta,R\alpha)\,\underline{d\theta}\,\underline{d\alpha}\\
+\int\lVert\Psi_a(\theta,\cdot)\rVert_\infty\int_{\mathclap{\lvert\theta-\alpha\rvert>\delta}}p(\alpha,R\theta)\,\underline{d\alpha}\,\underline{d\theta}\to0
\end{gather*}
as $R\uparrow1$ by Proposition \ref{prop:ii_1_9}(c) in conjunction with dominated convergence. It thereby follows that $\textstyle\iint_{\lvert\theta-\alpha\rvert>\delta}<\epsilon$ if we pick $R$ close enough to 1.

All in all, putting everything together, we have shown that
\begin{align*}
\lVert\int p(\alpha,R\theta)\Phi(\alpha,\cdot)\,\underline{d\alpha}-\Phi(\theta,\cdot)\mid\ell_{1,\infty}(\underline{d\theta},\R)\rVert<6\epsilon
\end{align*}
if we choose $R$ close enough to 1. It thus follows that $\textstyle\int\phi_e(u)p(\hat{u},R\theta)\cos(t\lvert u\rvert)\,du$ converges to $\textstyle\int_0^\infty\displaystyle\rho^{n-1}\phi_e(\theta\rho)\cos(t\rho)\,d\rho$ in $\ell_{1,\infty}(\underline{d\theta},dt)$ if $R\uparrow1.$
\end{proof}
Expanding the even and odd parts of $\phi$ in \eqref{eq:h} yields that
\begin{align*}
h(\theta,t)=\frac{1}{2}\int_0^\infty\rho^{n-1}\phi(\theta\rho)\exp(it\rho)\,d\rho+\frac{1}{2}\int_0^\infty\rho^{n-1}\phi(-\theta\rho)\exp(-it\rho)\,d\rho
\end{align*}
Now, $\lVert f\rVert_\sfL\leq\lVert f\rVert_\B$ is tantamount to $\lVert h\rVert_{1,\infty}\leq\lVert\phi\rVert_1,$ and plainly
\begin{align*}
\lVert h\rVert_{1,\infty}&\leq\frac{1}{2}\int\int_0^\infty\rho^{n-1}\lvert\phi(\theta\rho)\rvert\,d\rho\,\underline{d\theta}+\frac{1}{2}\int\int_0^\infty\rho^{n-1}\lvert\phi(-\theta\rho)\rvert\,d\rho\,\underline{d\theta}\\
&=\frac{1}{2}\int\,\lvert\phi(u)\rvert\,du+\frac{1}{2}\int\,\lvert\phi(-u)\rvert\,du=\lVert\phi\rVert_1
\end{align*}
\fbox{Proof that $\lVert f\rVert_\sfE\leq\lVert f\rVert_\sfL$ if $f\in\sfL$}

The crux of the argument is to show that $f_\ell(\theta,t)=\textstyle\int\displaystyle h_\ell(w,\langle w\mid t\theta\rangle)\,\underline{dw},$ which we do by expanding $f_\ell$ and $h_\ell$ w.r.t.\ the $\{S_{\ell j}\}_j$ from {\S\,}\ref{sec:zonal}. Ergo, letting
\begin{align*}
f_{\ell j}(t)\coloneq\int f_\ell(\theta,t)S_{\ell j}(\theta)\,\underline{d\theta}\overset{\text{F}}{=}\int f(t\alpha)\int S_{\ell j}(\theta)Z_\ell(\theta,\alpha)\,\underline{d\theta}\,\underline{d\alpha}=\int f(t\alpha)S_{\ell j}(\alpha)\,\underline{d\alpha}
\end{align*}
it follows that
\begin{align*}
\sum_j f_{\ell j}(t)S_{\ell j}(\theta)=\int f(t\alpha)\Bigl[\sum_j S_{\ell j}(\alpha)S_{\ell j}(\theta)\Bigr]\,\underline{d\alpha}=\int f(t\alpha)Z_\ell(\theta,\alpha)\,\underline{d\alpha}=f_\ell(\theta,t)
\end{align*}
\begin{lem}\label{lem:fl_ps_equiv}
$f_\ell(\theta,\cdot)\in\PS'$ for all $\theta$ is tantamount to $f_{\ell j}\in\PS'$ for all $j.$
\end{lem}
\begin{proof}
If $f_{\ell j}\in\PS'$ for all $j,$ then plainly $f_\ell(\theta,\cdot)=\sum_j S_{\ell j}(\theta)f_{\ell j}\in\PS'$ for all $\theta.$

Conversely, suppose $f_\ell(\theta,\cdot)\in\PS'$ for all $\theta.$ Picking $\alpha_1,\ldots,\alpha_N$ as in \eqref{eq:slj_matrix}:
\begin{align*}
\begin{Bmatrix}
S_{\ell,1}(\alpha_1)&\cdots&S_{\ell N}(\alpha_1)\\
\vdots&\ddots&\vdots\\
S_{\ell,1}(\alpha_N)&\cdots&S_{\ell N}(\alpha_N)
\end{Bmatrix}
^{\!-1}
\begin{Bmatrix}
f_\ell(\alpha_1,\cdot)\\
\vdots\\
f_\ell(\alpha_N,\cdot)
\end{Bmatrix}
=
\begin{Bmatrix}
f_{\ell,1}\\
\vdots\\
f_{\ell N}
\end{Bmatrix}
\end{align*}
where $N=N(\ell,n);$ ergo $f_{\ell j}\in\PS'$ for all $j.$
\end{proof}
\begin{lem}\label{lem:hlj}
$h_\ell(\theta,t)=\sum_{j=1}^{\mathclap{N(\ell,n)}} h_{\ell j}(t)S_{\ell j}(\theta)$ where $h_{\ell j}(t)\coloneq\int h_\ell(\theta,t)S_{\ell j}(\theta)\,\underline{d\theta}$
\end{lem}
\begin{proof}
We first consider the case $\ell\notin L.$ Lemma \ref{lem:fl_ps_equiv} allows us to conclude that
\begin{align*}
G_{\ell,\theta}\coloneq\text{D}_\ell\{f_\ell(\theta,\cdot)\}=\sum_i S_{\ell i}(\theta)\text{D}_\ell\{f_{\ell i}\}\eqcolon\sum_i S_{\ell i}(\theta)G_{\ell i}
\end{align*}
By Definition \ref{def:gamma_d}, $G_{\ell,\theta}[t]=G_\ell(\theta,t)\in\ell_{1,\infty}(\underline{d\theta},dt/t).$ Plugging in $\omega\in\sfP$ and integrating against $S_{\ell j}(\theta)\,\underline{d\theta}$ yields that
\begin{align*}
\int G_{\ell,\theta}(\omega)S_{\ell j}(\theta)\,\underline{d\theta}=\sum_i G_{\ell i}(\omega)\int S_{\ell i}(\theta)S_{\ell j}(\theta)\,\underline{d\theta}=G_{\ell j}(\omega)
\end{align*}
However, $G_{\ell,\theta}(\omega)=\textstyle\int_0^\infty\displaystyle G_{\ell,\theta}(t)\omega(t)\,dt$ and hence $G_{\ell j}(\omega)=$
\begin{align*}
\int G_{\ell,\theta}(\omega)S_{\ell j}(\theta)\,\underline{d\theta}=\int S_{\ell j}(\theta)\int_0^\infty G_{\ell,\theta}(t)\omega(t)\,dt\,\underline{d\theta}\overset{\text{F}}{=}\int_0^\infty\omega(t)\int G_\ell(\theta,t)S_{\ell j}(\theta)\,\underline{d\theta}\,dt
\end{align*}
so $G_{\ell j}[t]=\textstyle\int\displaystyle G_\ell(\theta,t)S_{\ell j}(\theta)\,\underline{d\theta}.$ By how $h_\ell$ is defined, it follows that
\begin{center}
$h_\ell(\theta,t)=\sum_j h_{\ell j}(t)S_{\ell j}(\theta)$ because $h_{\ell j}(t)=
\begin{dcases*}
G_{\ell j}(t)&if $t>0$ \\
(-1)^\ell G_{\ell j}(-t)&if $t<0$
\end{dcases*}$
\end{center}
\fbox{$\ell\in L$}

Let $m>t>0.$ By the dominated convergence theorem (as convergent sequences are bounded eo ipso)
\begin{align*}
\lim_{m\to\infty}\int_t^m f_{\ell j}(v)\,\frac{dv}{v}\overset{\text{F}}{=}\lim_{m\to\infty}\int S_{\ell j}(\theta)\int_t^m f_\ell(\theta,v)\,\frac{dv}{v}\,\underline{d\theta}=-\int S_{\ell j}(\theta)F_\ell(\theta,t)\,\underline{d\theta}
\end{align*}
so $F_{\ell j}(t)\coloneq\int S_{\ell j}(\theta)F_\ell(\theta,t)\,\underline{d\theta}=-\int_t^\to f_{\ell j}(v)\,\frac{dv}{v}$ and hence
\begin{align*}
\sum_j F_{\ell j}(t)S_{\ell j}(\theta)=-\int_t^\to\Bigl[\sum_j f_{\ell j}(v)S_{\ell j}(\theta)\Bigr]\,\frac{dv}{v}=-\int_t^\to f_\ell(\theta,v)\,\frac{dv}{v}=F_\ell(\theta,t)
\end{align*}
The rest of the argument follows mutatis mutandis like the case $\ell\notin L$ (with $\text{\DH}_\ell$ instead of $\text{D}_\ell$ and $F_\ell$ instead of $f_\ell$ etc.).
\end{proof}
Since $C_\ell^\alpha$ has the same parity as $\ell,$ \eqref{eq:f_in_terms_of_g} amounts to
\begin{align*}
f_\ell(\theta,t)=\lvert\sphere^{n-2}\rvert\int_{-1}^{+1}h_\ell(\theta,tv)C_\ell^{(n-2)/2}(v)(1-v^2)^{(n-3)/2}\,dv
\end{align*}
Now, using our newfound expansion of $h_\ell,$ Lemma \ref{lem:hlj}, and Lemma \ref{lem:funk_hecke},
\begin{align*}
\int h_\ell(w,\langle w\mid t\theta\rangle)\,\underline{dw}&=\sum_j\int h_{\ell j}(\langle w\mid t\theta\rangle)S_{\ell j}(w)\,\underline{dw}=\sum_j\int h_{\ell j}(t\langle w\mid\theta\rangle)S_{\ell j}(w)\,\underline{dw}\\
&=\sum_j S_{\ell j}(\theta)\lvert\sphere^{n-2}\rvert\int_{-1}^{+1} h_{\ell j}(tv)C_\ell^{(n-2)/2}(v)(1-v^2)^{(n-3)/2}\,dv\\
&=\lvert\sphere^{n-2}\rvert\int_{-1}^{+1}h_\ell(\theta,tv)C_\ell^{(n-2)/2}(v)(1-v^2)^{(n-3)/2}\,dv=f_\ell(\theta,t)
\end{align*}
For all $R\in(0,1),$
\begin{gather*}
\int f(t\alpha)p(\alpha,R\theta)\,\underline{d\alpha}\overset{\text{F}}{=}\sum_{\ell=0}^\infty R^\ell f_\ell(\theta,t)=\\
\sum_{\ell=0}^\infty R^\ell\int h_\ell(w,\langle w\mid t\theta\rangle)\,\underline{dw}\overset{\text{F}}{=}\int\sum_{\ell=0}^\infty R^\ell h_\ell(w,\langle w\mid t\theta\rangle)\,\underline{dw}
\end{gather*}
Taking $R\uparrow1,$ Proposition \ref{prop:ii_1_10} and the fact that $h_R\to h$ in $\ell_{1,\infty}(-\infty)$ yield that
\begin{align*}
f(t\theta)=\int h(w,\langle w\mid t\theta\rangle)\,\underline{dw}
\end{align*}
i.e., $f(x)=\textstyle\int\displaystyle h(w,\langle w\mid x\rangle)\,\underline{dw}$ for all $x\neq0$ (as $t>0).$ Since $\textstyle\int\displaystyle h(w,\langle w\mid\!{}\cdot{}\rangle)\,\underline{dw}$ is (sequentially) continuous by the dominated convergence theorem and $f\in C$ as well, this establishes that $\sfL\subset\sfE.$ The norm inequality is trivial.

\section{Barron spaces \& neural networks}\label{sec:main}

In this section we prove the main result of this paper and apply it to neural networks. Throughout this section $\alpha\in\N.$

\subsection{Main result}\label{sec:main_statement}

To state the main result of the paper, we need a couple more definitions.
\begin{defn}
The \myuline{Barron space} $\B^\alpha\coloneq\{f\in\B:\lvert u\rvert^\alpha\phi(u)\in\ell_1(du)\}$ where $\phi\in\ell_1$ is the unique function such that $f=\F\{\phi\}.$ We endow the Barron space with the following norm: $\lVert f\mid\B^\alpha\rVert=\lVert u\mapsto\lvert u\rvert^\alpha\phi(u)\mid\ell_1(du)\rVert.$
\end{defn}
The Barron spaces are clearly Banach spaces.
\begin{defn}\label{def:n_alpha}
The space $\sfE^\alpha$ consists of those functions $f\in C$ for which $\exists h^\alpha\in\ell_{1,\infty}(-\infty)$ such that, for all $r>0,$
\begin{align}\label{eq:f=expectation}
\forall x\in K(r):f(x)=\mathds{E}_{w,b}[2r\lvert\sphere^{n-1}\rvert h^\alpha(w,b)\delta^{(-\alpha)}(\langle w\mid x\rangle-b)]
\end{align}
where $w$ and $b$ are uniformly distributed on $\sphere^{n-1}$ on $[-r,r]$ resp. We endow the space $\sfE^\alpha$ with the following seminorm:
\begin{gather*}
\lVert f\mid\sfE^\alpha\rVert\coloneq\inf\Bigl\{\lVert h^\alpha\rVert_{1,\infty}:\eqref{eq:f=expectation}\text{ for all }r>0\Bigr\}
\end{gather*}
\end{defn}
While \eqref{eq:f=expectation} may be interpreted as an ``infinitely wide'' shallow neural network with random inner weights and biases (cf.\ {\S\,}\ref{sec:main_NN}), we can of course rewrite it as
\begin{align*}
\forall x\in K(r):f(x)=\iint h^\alpha(w,b)\delta^{(-\alpha)}(\langle w\mid x\rangle-b)[-r\leq b\leq r]\,db\,\underline{dw}
\end{align*}
so, compared to $\sfE,$ instead of integrating $h(w,\langle w\mid x\rangle)$ w.r.t.\ $w,$ we are now integrating $(h^\alpha(w,\cdot)[-r\leq\cdot\leq r]*\delta^{(-\alpha)})(\langle w\mid x\rangle)$ w.r.t.\ $w.$

The main result of the paper, plainly an analogue of Theorem \ref{thm:main_base_case}, may now be phrased as follows:
\begin{theorem}\label{thm:main}
$\B^\alpha\subset\sfL^{\!\alpha}\subset\sfE^\alpha$ with $\lVert f\mid\sfE^\alpha\rVert\leq\lVert f\mid\sfL^{\!\alpha}\rVert$ if $f\in\sfL^{\!\alpha}$ and $\lVert f\mid\sfL^{\!\alpha}\rVert\leq$ $\lVert f\mid\B^\alpha\rVert$ if $f\in\B^\alpha$
\end{theorem}

\subsection{Proof of Theorem \ref{thm:main}}

\fbox{Proof that $\lVert f\mid\sfL^{\!\alpha}\rVert\leq\lVert f\mid\B^\alpha\rVert$ if $f\in\B^\alpha$}

By definition, $\B^\alpha\subset\B,$ so Theorem \ref{thm:main_base_case} yields that $f\in\B^\alpha\subset\B\subset\sfL;$ it follows that $G_\ell$ and $h_\ell$ as defined in Definition \ref{def:gamma_d} are still well-defined.
\begin{lem}
If $f\in\sfL,$ then $\M_\alpha^{-1}\Bigl\{N_\ell^\alpha\M_0\{f_\ell(\theta,\cdot)\}\Bigr\}=\partial^\alpha G_\ell(\theta,\cdot).$
\end{lem}
\begin{proof}
If $\ell\notin L,$ this is the content of Corollary \ref{cor:d_decomp}, as $G_\ell\in\ell_{1,\infty}(0)\Rightarrow G_\ell(\theta,\cdot)\in$ $\ell_\infty(0,\infty)\subset\PS'$ for almost every $\theta.$ If $\ell\in L$ instead, \eqref{eq:partial_alpha} and \eqref{eq:F->f} yield that
\begin{gather*}
\partial^\alpha G_\ell(\theta,\cdot)=\M_\alpha^{-1}\biggl\{\frac{\Gamma(1-iy)}{\Gamma(1-iy-\alpha)}\M_0\{G_\ell(\theta,\cdot)\}[y]\biggr\}\\
=\M_\alpha^{-1}\biggl\{\frac{-iy\Gamma(-iy)}{\Gamma(1-iy-\alpha)}N_\ell(y)\M_0\{F_\ell(\theta,\cdot)\}[y]\biggr\}\\
=\M_\alpha^{-1}\biggl\{\frac{\Gamma(-iy)}{\Gamma(1-iy-\alpha)}N_\ell(y)\M_0\{f_\ell(\theta,\cdot)\}[y]\biggr\}=\M_\alpha^{-1}\Bigl\{N_\ell^\alpha\M_0\{f_\ell(\theta,\cdot)\}\Bigr\}\tag*{\qedhere}
\end{gather*}
\end{proof}
It follows that $\partial^\alpha G_\ell(\theta,\cdot)=G_\ell^\alpha(\theta,\cdot)\in\ell_{1,\infty}(\underline{d\theta},\R)$ and hence $\partial^\alpha h_\ell(\theta,t)=h_\ell^\alpha(\theta,t)$ for $t>0.$ If $t<0,$ then $\partial^\alpha h_\ell(\theta,t)=\partial^\alpha[(-1)^\ell G_\ell(\theta,-t)]=(-1)^{\ell+\alpha}G_\ell^\alpha(\theta,-t)=$ $h_\ell^\alpha(\theta,t)$ as well, so $\partial^\alpha h_\ell(\theta,\cdot)=h_\ell^\alpha(\theta,\cdot).$

Let $\phi^\circ(u)\coloneq\lvert u\rvert^\alpha\phi(u)\in\ell_1(du).$ Then
\begin{align*}
h_\ell^\alpha(\theta,t)=\partial^\alpha h_\ell(\theta,t)&=\partial^\alpha\int\phi(u)Z_\ell(\theta,\hat{u})\cis_\ell(t\lvert u\rvert)\,du\\
&=\int\phi(u)Z_\ell(\theta,\hat{u})\partial^\alpha[\cis_\ell(t\lvert u\rvert)]\,du\\
&=\int\phi^\circ(u)Z_\ell(\theta,\hat{u})\cis_\ell(t\lvert u\rvert+\alpha\tfrac{\pi}{2})\,du
\end{align*}
where the exchange of differentiation and integration is justified by dominated convergence.

We claim $\lVert h^\alpha\rVert_{1,\infty}\leq\lVert\phi^\circ\rVert_1.$ If $\alpha\equiv0\,\%\,4,$ the argument is the same as in the proof of Theorem \ref{thm:main_base_case}, and the reasoning is analogous when $\alpha\not\equiv0\,\%\,4;$ e.g., if $\alpha\equiv1\,\%\,4,$ the parity of $h_\ell^\alpha(\theta,\cdot)$ is opposite the parity of $\ell,$ so
\begin{align*}
h_\ell^\alpha(\theta,t)=-\int\phi_e^\circ(u)Z_\ell(\theta,\hat{u})\sin(t\lvert u\rvert)\,du
\end{align*}
for even $\ell,$ while for odd $\ell$ the corresponding formula reads
\begin{align*}
h_\ell^\alpha(\theta,t)=i\int\phi_o^\circ(u)Z_\ell(\theta,\hat{u})\cos(t\lvert u\rvert)\,du
\end{align*}
Like before, $\textstyle\sum_{\ell=0}^\infty R^\ell\lVert h_\ell^\alpha\rVert_{1,\infty}<\infty
$ for $R\in(0,1),$ so $h_R^\alpha(\theta,t)\coloneq\textstyle\sum_{\ell=0}^\infty R^\ell h_\ell^\alpha(\theta,t)=$
\begin{align*}
i\sum_{2\nmid\ell}R^\ell\int\phi_o^\circ(u)Z_\ell(\theta,\hat{u})\cos(t\lvert u\rvert)\,du&-\sum_{2\mid\ell}R^\ell\int\phi_e^\circ(u)Z_\ell(\theta,\hat{u})\sin(t\lvert u\rvert)\,du\\
=i\sum_{\ell=0}^\infty R^\ell\int\phi_o^\circ(u)Z_\ell(\theta,\hat{u})\cos(t\lvert u\rvert)\,du&-\sum_{\ell=0}^\infty R^\ell\int\phi_e^\circ(u)Z_\ell(\theta,\hat{u})\sin(t\lvert u\rvert)\,du\\
\overset{\text{F}}{=}i\int\phi_o^\circ(u)p(\hat{u},R\theta)\cos(t\lvert u\rvert)\,du&-\int\phi_e^\circ(u)p(\hat{u},R\theta)\sin(t\lvert u\rvert)\,du
\end{align*}
Mimicking Lemma \ref{lem:h} and expanding the even and odd parts of $\phi^\circ$ yields that
\begin{align*}
h^\alpha(\theta,t)&=i\int_0^\infty\rho^{n-1}\phi_o^\circ(\theta\rho)\cos(t\rho)\,d\rho-\int_0^\infty\rho^{n-1}\phi_e^\circ(\theta\rho)\sin(t\rho)\,d\rho\\
&=\frac{i}{2}\int_0^\infty\rho^{n-1}\phi^\circ(\theta\rho)\exp(it\rho)\,d\rho-\frac{i}{2}\int_0^\infty\rho^{n-1}\phi^\circ(-\theta\rho)\exp(-it\rho)\,d\rho
\end{align*}
from which $\lVert h^\alpha\rVert_{1,\infty}\leq\lVert\phi^\circ\rVert_1$ follows like in the proof of Theorem \ref{thm:main_base_case}.

All in all, we have thus shown that
\begin{align*}
\lVert f\mid\sfL^{\!\alpha}\rVert=\lVert h^\alpha\rVert_{1,\infty}\leq\lVert\phi^\circ\rVert_1=\lVert f\mid\B^\alpha\rVert
\end{align*}
\fbox{Proof that $\lVert f\mid\sfE^\alpha\rVert\leq\lVert f\mid\sfL^{\!\alpha}\rVert$ if $f\in\sfL^{\!\alpha}$}

We first define
\begin{align*}
h_\ell^\circ(\theta,t)\coloneq[t\leq r]\int h_\ell^\alpha(\theta,b)[-r\leq b]\delta^{(-\alpha)}(t-b)\,db
\end{align*}
Both $\delta^{(-\alpha)}$ and $b\mapsto h_\ell^\alpha(\theta,b)[-r\leq b]$ have supports bounded on the left so their convolution is well-defined. The crux of the proof is that
\begin{align}\label{eq:fhcirc}
f_\ell(\theta,t)=\int h_\ell^\circ(w,\langle w\mid t\theta\rangle)\,\underline{dw}\qquad\text{for }t\in(0,r]
\end{align}
To show \eqref{eq:fhcirc} we will need to distinguish three cases, but first we show that \eqref{eq:fhcirc} yields the desideratum.
\begin{lem}\label{lem:h_circ_norm_bound}
$\lVert h_\ell^\circ\rVert_{1,\infty}\leq\lVert h_\ell^\alpha\rVert_{1,\infty}(2r)^\alpha/\alpha!$
\end{lem}
\begin{proof}
Plainly,
\begin{gather*}
\lvert\int\displaystyle h_\ell^\alpha(\theta,b)[-r\leq b]\delta^{(-\alpha)}(t-b)\,db\rvert\leq\lVert h_\ell^\alpha(\theta,\cdot)\rVert_\infty\int\delta^{(-\alpha)}(t-b)[-r\leq b]\,db\\
\begin{align*}
\Rightarrow\lVert h_\ell^\circ\rVert_{1,\infty}&\leq\lVert h_\ell^\alpha\rVert_{1,\infty}\lVert[t\leq r]\textstyle\int\displaystyle\delta^{(-\alpha)}(t-b)[-r\leq b]\,db\mid\ell_\infty(dt)\rVert\\
&=\lVert h_\ell^\alpha\rVert_{1,\infty}\lVert[-r\leq t\leq r](r+t)^\alpha/\alpha!\mid\ell_\infty(dt)\rVert\\
&=\lVert h_\ell^\alpha\rVert_{1,\infty}(2r)^\alpha/\alpha!\tag*{\qedhere}
\end{align*}
\end{gather*}
\end{proof}
\begin{lem}\label{lem:h_circ_aas}
$\{h_\ell^\circ\}_{\ell=0}^\infty\in\AAS(-\infty)$
\end{lem}
\begin{proof}
Let $R\in(0,1).$ By Lemma \ref{lem:h_circ_norm_bound} in conjunction with the fact that
\begin{align*}
\{h_\ell^\alpha\}_{\ell=0}^\infty\in\AAS(-\infty)\Rightarrow\sum_{\ell=0}^\infty R^\ell\lVert h_\ell^\alpha\rVert_{1,\infty}<\infty
\end{align*}
it follows that $h_R^\circ\coloneq\textstyle\sum_{\ell=0}^\infty\displaystyle R^\ell h_\ell^\circ\in\ell_{1,\infty}(-\infty)$ exists. In fact,
\begin{align*}
h_R^\circ(\theta,t)\overset{\text{F}}{=}[t\leq r]\int h_R^\alpha(\theta,b)[-r\leq b]\delta^{(-\alpha)}(t-b)\,db
\end{align*}
for almost every $\theta$ since $h_R^\alpha(\theta,\cdot)\in\ell_\infty$ for almost every $\theta.$ Dominated convergence now yields that
\begin{align}\label{eq:h_circ_h_alpha}
h^\circ(\theta,t)=[t\leq r]\int h^\alpha(\theta,b)[-r\leq b]\delta^{(-\alpha)}(t-b)\,db
\end{align}
so $\lVert h^\circ\rVert_{1,\infty}\leq\lVert h^\alpha\rVert_{1,\infty}(2r)^\alpha/\alpha!<\infty$ analogously to Lemma \ref{lem:h_circ_norm_bound}.
\end{proof}

\begin{lem}\label{lem:final}
If $f_\ell(\theta,t)=\textstyle\int\displaystyle h_\ell^\circ(w,\langle w\mid t\theta\rangle)\,\underline{dw}$ for $t\in(0,r]$ as in \eqref{eq:fhcirc}, then
\begin{align*}
\forall x\in K(r):f(x)&=\int h^\circ(w,\langle w\mid x\rangle)\,\underline{dw}\\
&=\iint h^\alpha(w,b)\delta^{(-\alpha)}(\langle w\mid x\rangle-b)[-r\leq b\leq r]\,db\,\underline{dw}
\end{align*}
where $h^\circ\coloneq\textstyle\osum_{\ell=0}^\infty\displaystyle h_\ell^\circ.$
\end{lem}
\begin{proof}
Letting $f_\ell^\circ(\theta,t)\coloneq\textstyle\int\displaystyle h_\ell^\circ(w,\langle w\mid t\theta\rangle)\,\underline{dw},$ Lemma \ref{lem:h_circ_aas} facilitates that
\begin{align*}
\sum_{\ell=0}^\infty R^\ell f_\ell^\circ(\theta,t)\overset{\text{F}}{=}\int\sum_{\ell=0}^\infty R^\ell h_\ell^\circ(w,\langle w\mid t\theta\rangle)\,\underline{dw}=\int h_R^\circ(w,\langle w\mid t\theta\rangle)\,\underline{dw}
\end{align*}
for all $R\in(0,1).$ Dominated convergence (again facilitated by Lemma \ref{lem:h_circ_aas}) now yields that
\begin{align*}
f^\circ(\theta,t)\coloneq\osum_{\ell=0}^\infty f_\ell^\circ(\theta,t)=\int h^\circ(w,\langle w\mid t\theta\rangle)\,\underline{dw}
\end{align*}
Plainly $f^\circ\in\ell_{1,\infty}(0)\supset\{f_\ell^\circ\}_{\ell=0}^\infty,$ so $\{f_\ell^\circ\}_{\ell=0}^\infty\in\AAS(0).$

While $t\leq r,$ \eqref{eq:h_circ_h_alpha} may be rewritten as
\begin{align*}
h^\circ(\theta,t)=\int h^\alpha(\theta,b)\delta^{(-\alpha)}(t-b)[-r\leq b\leq r]\,db
\end{align*}
so we may conclude that
\begin{align*}
f^\circ(\theta,t)=\int h^\circ(w,\langle w\mid t\theta\rangle)\,\underline{dw}=\iint h^\alpha(w,b)\delta^{(-\alpha)}(\langle w\mid t\theta\rangle-b)[-r\leq b\leq r]\,db\,\underline{dw}
\end{align*}
which implies the desideratum for almost all $x\in K(r)\setminus\{0\}$ if $f^\circ(\theta,t)=f(t\theta)$ for almost all $\theta$ while $t\in(0,r].$

By Lemma \ref{lem:h_circ_aas}, $h^\circ\in\ell_{1,\infty}(-\infty),$ so $\textstyle\int\displaystyle h^\circ(w,\langle w\mid\!{}\cdot{}\rangle)\,\underline{dw}$ is (sequentially) continuous by the dominated convergence theorem. Since $f\in C$ as well, $f^\circ(\theta,t)=f(t\theta)$ for almost every $\theta$ while $t\in(0,r]$ in fact implies the desideratum $\forall x\in K(r)$ because conull sets are dense (contraposition readily reveals why).

All that remains is therefore to show that $f^\circ(\theta,t)=f(t\theta)$ for almost all $\theta$ while $t\in(0,r]$ is arbitrarily fixed.

Let $t\in(0,r].$ Since $\{f_\ell^\circ\}_{\ell=0}^\infty\in\AAS(0),$ it follows that $\{f_{\ell\mid t}^\circ\}_{\ell=0}^\infty$ is absolutely Abel summable in $\ell_1(\sphere^{n-1}),$ where $f_{\ell\mid t}^\circ(\theta)\coloneq f_\ell^\circ(\theta,t).$ By assumption, $f_{\ell\mid t}^\circ=f_{\ell\mid t},$ so
\begin{align*}
f_{\cdot\mid t}^\circ=\osum_{\ell=0}^\infty f_{\ell\mid t}^\circ=\osum_{\ell=0}^\infty f_{\ell\mid t}=f_{\cdot\mid t}
\end{align*}
in $\ell_1(\sphere^{n-1})$ by Theorem \ref{thm:a_undoes_z}, where $f_{\cdot\mid t}^\circ(\theta)\coloneq f^\circ(\theta,t).$ This is exactly what remained to be shown.
\end{proof}
Through Lemma \ref{lem:final}, \eqref{eq:fhcirc} is readily seen to yield the desideratum, so the time has come to prove \eqref{eq:fhcirc}. Of the three cases, we start with $\ell\notin L.$

Let $G_{\ell,\theta}\coloneq\text{D}_\ell\{f_\ell(\theta,\cdot)\}\in\sfP'.$ By Corollary \ref{cor:d_decomp}, $G_\ell^\alpha(\theta,t)=\partial^\alpha G_{\ell,\theta}[t],$ so $G_{\ell,\theta}$ is locally bounded by Lemma \ref{lem:distributional_antiderivative}. If $t<0,$ we thus have
\begin{gather*}
\partial^\alpha h_{\ell,\theta}(t)=\partial^\alpha[(-1)^\ell G_{\ell,\theta}(-t)]=(-1)^{\ell+\alpha}G_\ell^\alpha(\theta,-t)=h_\ell^\alpha(\theta,t)\\
\text{where }h_{\ell,\theta}(t)\coloneq\begin{dcases*}
G_{\ell,\theta}(t)&if $t>0$\\
(-1)^\ell G_{\ell,\theta}(-t)&if $t<0$
\end{dcases*}
\end{gather*}
Consequently, $h_{\ell,\theta}$ is locally bounded and $h_\ell^\alpha(\theta,\cdot)=\partial^\alpha h_\ell(\theta,\cdot)\eqcolon h_{\ell,\theta}^{(\alpha)}.$
\begin{lem}\label{lem:h_circ_h}
$h_\ell^\circ(\theta,t)=h_{\ell,\theta}(t)[-r\leq t\leq r]$ when $\ell\notin L$ for all $r>0$ and $\alpha\in\N$
\end{lem}
\begin{proof}
It suffices to show that
\begin{align*}
\int h_{\ell,\theta}^{(\alpha)}(b)[-r\leq b]\delta^{(-\alpha)}(t-b)\,db=h_{\ell,\theta}(t)[-r\leq t]
\end{align*}
We can use integration by parts to derive that
\begin{align*}
\int h_{\ell,\theta}^{(\alpha)}(b)[-r\leq b]\delta^{(-\alpha)}(t-b)\,db=\int h_{\ell,\theta}^{(\alpha-1)}(b)[-r\leq b]\delta^{(1-\alpha)}(t-b)\,db
\end{align*}
because $\partial\{h_{\ell,\theta}^{(\alpha-1)}(t)[-r\leq t]\}=h_{\ell,\theta}^{(\alpha)}(t)[-r\leq t]$ except at $t=-r.$ Hence $h_{\ell,\theta}^{(\alpha-1)}$ is differentiable (and thus continuous) on $(-r,\infty)$ which means the integral on the right is well-defined. Repeatedly integrating by parts yields that
\begin{align*}
\int h_{\ell,\theta}^{(\alpha)}(b)[-r\leq b]\delta^{(-\alpha)}(t-b)\,db&=\int h_{\ell,\theta}^{(1)}(\theta,b)[-r\leq b]\delta^{(-1)}(t-b)\,db\\
&=h_{\ell,\theta}(t)[-r\leq t]
\end{align*}
by the fundamental theorem of calculus.
\end{proof}

Let $f_\ell^\circ(\theta,t)\coloneq\lvert\sphere^{n-2}\rvert\int_{-1}^{+1}h_\ell^\circ(\theta,tv)C_\ell^{(n-2)/2}(v)(1-v^2)^{(n-3)/2}\,dv.$
\begin{lem}\label{lem:f=fcirc}
$f_\ell^\circ(\theta,\cdot)=f_\ell(\theta,\cdot)$ on $(0,r]$ whenever $\ell\notin L$
\end{lem}
\begin{proof}
If $t\in(0,r],$ then $tv\in[-r,r]$ while $v\in[-1,1],$ so, by Lemma \ref{lem:h_circ_h},
\begin{align*}
f_\ell^\circ(\theta,t)&=\lvert\sphere^{n-2}\rvert\int_{-1}^{+1}h_\ell^\circ(\theta,tv)C_\ell^{(n-2)/2}(v)(1-v^2)^{(n-3)/2}\,dv\\
&=\lvert\sphere^{n-2}\rvert\int_{-1}^{+1}h_{\ell,\theta}(tv)C_\ell^{(n-2)/2}(v)(1-v^2)^{(n-3)/2}\,dv\\
&=2\lvert\sphere^{n-2}\rvert\int_0^1 G_{\ell,\theta}(tv)C_\ell^{(n-2)/2}(v)(1-v^2)^{(n-3)/2}\,dv=f_\ell(\theta,t)
\end{align*}
The last equality follows from Lemma \ref{lem:d_inverted} and the penultimate inequality follows from the parity of $C_\ell^\alpha.$
\end{proof}
In light of Lemma \ref{lem:f=fcirc}, we need only prove that $f_\ell^\circ(\theta,t)=\textstyle\int\displaystyle h_\ell^\circ(w,\langle w\mid t\theta\rangle)\,\underline{dw}$ for $t\in(0,r].$
\begin{lem}\label{lem:hljcirc}
Let $\ell\notin L.$ While $t\in[-r,r],$
\begin{center}
$h_\ell^\circ(\theta,t)=\sum_{j=1}^{\mathclap{N(\ell,n)}} h_{\ell j}^\circ(t)S_{\ell j}(\theta)$ with $h_{\ell j}^\circ(t)\coloneq\int h_\ell^\circ(\theta,t)S_{\ell j}(\theta)\,\underline{d\theta}$
\end{center}
\end{lem}
\begin{proof}
Lemma \ref{lem:fl_ps_equiv} still allows us to conclude that
\begin{align*}
G_{\ell,\theta}=\text{D}_\ell\{f_\ell(\theta,\cdot)\}=\sum_i S_{\ell i}(\theta)\text{D}_\ell\{f_{\ell i}\}\eqcolon\sum_i S_{\ell i}(\theta)G_{\ell i}
\end{align*}
By  Lemma \ref{lem:h_circ_h} (and Lemma \ref{lem:h_circ_aas}), $\theta\mapsto\lVert G_{\ell,\theta}\mid\ell_\infty(0,r]\rVert=\lVert h_{\ell,\theta}\mid\ell_\infty[-r,r]\rVert$ is integrable. Plugging in $\omega\in\sfP$ \myuline{with support contained in} $(0,r]$ and subsequently integrating against $S_{\ell j}(\theta)\,\underline{d\theta}$ yields that
\begin{align*}
\int G_{\ell,\theta}(\omega)S_{\ell j}(\theta)\,\underline{d\theta}=\sum_i G_{\ell i}(\omega)\int S_{\ell i}(\theta)S_{\ell j}(\theta)\,\underline{d\theta}=G_{\ell j}(\omega)
\end{align*}
However, $G_{\ell,\theta}(\omega)=\textstyle\int_0^\infty\displaystyle G_{\ell,\theta}(t)\omega(t)\,dt$ and hence $G_{\ell j}(\omega)=$
\begin{align*}
\int G_{\ell,\theta}(\omega)S_{\ell j}(\theta)\,\underline{d\theta}=\int S_{\ell j}(\theta)\int_0^\infty G_{\ell,\theta}(t)\omega(t)\,dt\,\underline{d\theta}\overset{\text{F}}{=}\int_0^\infty\omega(t)\int G_\ell(\theta,t)S_{\ell j}(\theta)\,\underline{d\theta}\,dt
\end{align*}
so $G_{\ell j}[t]=\textstyle\int\displaystyle G_\ell(\theta,t)S_{\ell j}(\theta)\,\underline{d\theta}.$ By how $h_{\ell,\theta}$ is defined, it follows that
\begin{center}
$h_{\ell,\theta}=\sum_j S_{\ell j}(\theta)h_{\ell j}$ with $h_{\ell j}(t)\coloneq\int h_{\ell,\theta}(t)S_{\ell j}(\theta)\,\underline{d\theta}=
\begin{dcases*}
G_{\ell j}(t)&if $t>0$ \\
(-1)^\ell G_{\ell j}(-t)&if $t<0$
\end{dcases*}$
\end{center}
The desideratum now follows from Lemma \ref{lem:h_circ_h} because
\begin{align*}
h_{\ell j}^\circ(t)=\int h_\ell^\circ(\theta,t)S_{\ell j}(\theta)\,\underline{d\theta}=h_{\ell j}(t)[-r\leq t\leq r]\tag*{\qedhere}
\end{align*}
\end{proof}
\begin{lem}
$f_\ell^\circ(\theta,t)=\textstyle\int\displaystyle h_\ell^\circ(w,\langle w\mid t\theta\rangle)\,\underline{dw}$ for $t\in(0,r]$ and $\ell\notin L$
\end{lem}
\begin{proof}
Lemma \ref{lem:hljcirc} and Lemma \ref{lem:funk_hecke} yield that
\begin{gather}
\int h_\ell^\circ(w,\langle w\mid t\theta\rangle)\,\underline{dw}=\sum_j\int h_{\ell j}^\circ(t\langle w\mid\theta\rangle)S_{\ell j}(w)\,\underline{dw}\notag\\
=\sum_j S_{\ell j}(\theta)\lvert\sphere^{n-2}\rvert\int_{-1}^1 h_{\ell j}^\circ(tv)C_\ell^{(n-2)/2}(v)(1-v^2)^{(n-3)/2}\,dv\label{eq:final}\\
=\lvert\sphere^{n-2}\rvert\int_{-1}^{+1}h_\ell^\circ(\theta,tv)C_\ell^{(n-2)/2}(v)(1-v^2)^{(n-3)/2}\,dv=f_\ell^\circ(\theta,t)\tag*{\qedhere}
\end{gather}
\end{proof}

\fbox{$\ell\in L,\alpha=1$}

Lemma \ref{lem:dh_decomp} tells us that $\text{\DH}_\ell^1=t^{-1}\text{\DH}_\ell,$ so
\begin{align}\label{eq:gltheta}
G_{\ell,\theta}[t]\coloneq\text{\DH}_\ell\{f_\ell(\theta,\cdot)\}[t]=G_\ell^1(\theta,t)t
\end{align}
and hence, by Lemma \ref{lem:d_inverted},
\begin{align*}
f_\ell(\theta,t)&=2\lvert\sphere^{n-2}\rvert\int_0^1 G_\ell^1(\theta,tv)(tv)\tfrac{1}{n-1}C_{\ell-1}^{n/2}(v)(1-v^2)^{(n-1)/2}\,\frac{dv}{v}\\
&=2t\lvert\sphere^{n-2}\rvert\int_0^1 G_\ell^1(\theta,tv)\tfrac{1}{n-1}C_{\ell-1}^{n/2}(v)(1-v^2)^{(n-1)/2}\,dv\\
&=t\lvert\sphere^{n-2}\rvert\int_{-1}^{+1}h_\ell^1(\theta,tv)\tfrac{1}{n-1}C_{\ell-1}^{n/2}(v)(1-v^2)^{(n-1)/2}\,dv
\end{align*}
Let $f_\ell^\circ(\theta,t)\coloneq\lvert\sphere^{n-2}\rvert\int_{-1}^{+1}h_\ell^\circ(\theta,tv)C_\ell^{(n-2)/2}(v)(1-v^2)^{(n-3)/2}\,dv$ again.
\begin{lem}\label{lem:f_circ_f}
$f_\ell^\circ(\theta,\cdot)=f_\ell(\theta,\cdot)$ on $(0,r]$ when $\ell\in L,\alpha=1$
\end{lem}
\begin{proof}
Let $t\in(0,r].$ Then $\lvert z\rvert\leq 1\Rightarrow\lvert tz\rvert\leq r.$ Ditto with $z$ replaced by $v.$ Ergo:
\begin{gather*}
f_\ell(\theta,t)=t\lvert\sphere^{n-2}\rvert\int h_\ell^1(\theta,tz)[-r\leq tz\leq r]\tfrac{1}{n-1}C_{\ell-1}^{n/2}(z)(1-z^2)^{(n-1)/2}[-1\leq z\leq 1]\,dz\\
=t\lvert\sphere^{n-2}\rvert\int_{\mathclap{-\infty}}^\infty h_\ell^1(\theta,tz)[-r\leq tz\leq r]\int_z^{\infty}C_\ell^{(n-2)/2}(v)(1-v^2)^{(n-3)/2}[-1\leq v\leq 1]\,dv\,dz\\
\overset{\text{F}}{=}t\lvert\sphere^{n-2}\rvert\int_{\mathclap{-\infty}}^\infty C_\ell^{(n-2)/2}(v)(1-v^2)^{(n-3)/2}[-1\leq v\leq 1]\int_{\mathclap{-\infty}}^v h_\ell^1(\theta,tz)[-r\leq tz\leq r]\,dz\,dv\\
=\lvert\sphere^{n-2}\rvert\int_{-1}^{+1}C_\ell^{(n-2)/2}(v)(1-v^2)^{(n-3)/2}\int_{\mathclap{-\infty}}^{tv} h_\ell^1(\theta,b)[-r\leq b\leq r]\,db\,dv\\
=\lvert\sphere^{n-2}\rvert\int_{-1}^{+1}C_\ell^{(n-2)/2}(v)(1-v^2)^{(n-3)/2}[tv\leq r]\int h_\ell^1(\theta,b)[-r\leq b\leq tv]\,db\,dv\\
=\lvert\sphere^{n-2}\rvert\int_{-1}^{+1}h_\ell^\circ(\theta,tv)C_\ell^{(n-2)/2}(v)(1-v^2)^{(n-3)/2}\,dv=f_\ell^\circ(\theta,t)
\end{gather*}
where the last two equalities follow from how we defined $h_\ell^\circ$ and $f_\ell^\circ.$
\end{proof}
In light of Lemma \ref{lem:f_circ_f}, it suffices to prove that $f_\ell^\circ(\theta,t)=\textstyle\int\displaystyle h_\ell^\circ(w,\langle w\mid t\theta\rangle)\,\underline{dw}$ for $t\in(0,r].$

\begin{lem}\label{lem:f_circ_h_circ}
$f_\ell^\circ(\theta,t)=\textstyle\int\displaystyle h_\ell^\circ(w,\langle w\mid t\theta\rangle)\,\underline{dw}$ for $t\in(0,r]$ and $\ell\in L,\alpha=1$
\end{lem}
\begin{proof}
By Definition \ref{def:adz}, $G_\ell^1\in\ell_{1,\infty}(0),$ whence $\theta\mapsto\lVert G_{\ell,\theta}\mid\ell_\infty(0,r]\rVert$ is integrable (recall \eqref{eq:gltheta}). We may thus repeat the bulk of the proof of Lemma \ref{lem:hljcirc} mutatis mutandis (replacing every instance of $\text{D}_\ell$ with $\text{\DH}_\ell$ etc.) to obtain that
\begin{center}
$h_\ell^1(\theta,t)t=\sum_j S_{\ell j}(\theta)h_{\ell j}^1(t)t$ where $h_{\ell j}^1(t)\coloneq\int h_\ell^1(\theta,t)S_{\ell j}(\theta)\,\underline{d\theta}$
\end{center}
Dividing by $t>0$ thus allows us to conclude that
\begin{gather*}
h_\ell^\circ(\theta,t)=[t\leq r]\int h_\ell^1(\theta,b)[-r\leq b\leq t]\,db\\
=\sum_j S_{\ell j}(\theta)[t\leq r]\int h_{\ell j}^1(b)[-r\leq b\leq t]\,db\eqcolon\sum_j S_{\ell j}(\theta)h_{\ell j}^\circ(t)
\end{gather*}
Lastly, the computations in \eqref{eq:final} may be repeated verbatim.
\end{proof}

\fbox{$\ell\in L,\alpha>1$}

This final case combines ideas from both of the previous cases.

Let $G_{\ell,\theta}^1\coloneq t^{-1}\text{\DH}_\ell\{f_\ell(\theta,\cdot)\}\in\sfP'.$ Lemma \ref{lem:dh_decomp} yields that $G_\ell^\alpha(\theta,t)=\partial^{\alpha-1}G_{\ell,\theta}^1[t],$ whence $G_{\ell,\theta}^1$ is locally bounded by Lemma \ref{lem:distributional_antiderivative}. If $t<0,$ we thus have
\begin{gather*}
\partial^{\alpha-1}h_{\ell,\theta}^1(t)=(-1)^\ell\partial^{\alpha-1}[G_{\ell,\theta}^1(-t)]=(-1)^{\ell+\alpha}G_\ell^\alpha(\theta,-t)=h_\ell^\alpha(\theta,t)\\
\text{where }h_{\ell,\theta}^1(t)\coloneq\begin{dcases*}
G_{\ell,\theta}^1(t)&if $t>0$\\
(-1)^{\ell+1}G_{\ell,\theta}^1(-t)&if $t<0$
\end{dcases*}
\end{gather*}
Consequently $h_{\ell,\theta}^1$ is locally bounded and $h_\ell^\alpha(\theta,\cdot)=\partial^{\alpha-1}h_{\ell,\theta}^1\eqcolon(h_{\ell,\theta}^1)^{(\alpha-1)}.$
\begin{lem}\label{lem:h_circ_h1}
$h_\ell^\circ(\theta,t)=[t\leq r]\int h_{\ell,\theta}^1(b)[-r\leq b\leq t]\,db$ whenever $\ell\in L,\alpha>1$
\end{lem}
\begin{proof}
It suffices to show that
\begin{align*}
\int\delta^{(-\alpha)}(t-b)(h_{\ell,\theta}^1)^{(\alpha-1)}(b)[-r\leq b]\,db=\int h_{\ell,\theta}^1(b)[-r\leq b\leq t]\,db
\end{align*}
because multiplying by $[t\leq r]$ on both sides yields the desideratum. We can use integration by parts to derive that
\begin{align*}
\int\delta^{(-\alpha)}(t-b)(h_{\ell,\theta}^1)^{(\alpha-1)}(b)[-r\leq b]\,db=\int\delta^{(1-\alpha)}(t-b)(h_{\ell,\theta}^1)^{(\alpha-2)}(b)[-r\leq b]\,db
\end{align*}
as $\partial\{(h_{\ell,\theta}^1)^{(\alpha-2)}(t)[-r\leq t]\}=(h_{\ell,\theta}^1)^{(\alpha-1)}(t)[-r\leq t]$ except at $t=-r,$ so $(h_{\ell,\theta}^1)^{(\alpha-2)}$ is differentiable (and thus continuous) on $(-r,\infty).$ Hence the integral on the right is well-defined. Repeatedly integrating by parts yields that
\begin{gather*}
\int\delta^{(-\alpha)}(t-b)(h_{\ell,\theta}^1)^{(\alpha-1)}(b)[-r\leq b]\,db=\int\delta^{(-1)}(t-b)h_{\ell,\theta}^1(\theta,b)[-r\leq b]\,db\\
=\int h_{\ell,\theta}^1(b)[-r\leq b\leq t]\,db\tag*{\qedhere}
\end{gather*}
\end{proof}
Yet again, let $f_\ell^\circ(\theta,t)\coloneq\lvert\sphere^{n-2}\rvert\int_{-1}^{+1}h_\ell^\circ(\theta,tv)C_\ell^{(n-2)/2}(v)(1-v^2)^{(n-3)/2}\,dv.$

The rest of the proof is mutatis mutandis (replace $G_\ell^1(\theta,\cdot)$ with $G_{\ell,\theta}^1$ etc.) the \\ proof for $\alpha=1.$

\subsection{Random neural networks}\label{sec:main_NN}

In this section we consider approximating $f\in\sfE^\alpha$ on $K\coloneq K(r)$ by a random neural network on a rich enough probability space $(\Omega,\Sigma,\P)$ of the form
\begin{align*}
f_m(x)=\sum_{j=1}^m a_j\delta^{(-\alpha)}(\langle w_j\mid x\rangle-b_j)
\end{align*}
where the $(w_j,b_j)$ are independent random draws from the probability law $H$ with pdf $\lvert h\rvert\div\lVert h\rVert_1$ (supposing $\lVert h\rVert_1>0)$ and
\begin{align*}
h:\sphere^{n-1}\times\R\ni(w,b)\mapsto h^\alpha(w,b)[-r\leq b\leq r]
\end{align*}
In this section we will not use $h$ (the $f_\ell)$ from earlier, so in this section $h$ (the $f_m)$ \\ is (are) always as in the (pen)ultimate display. We will motivate our choice of $f_m$ (and $h)$ momentarily.

Specifically, we want to know how large $m$ needs to be to have $\lVert f-f_m\rVert_{\infty,K}\leq\epsilon$ w.h.p., where $\lVert\cdot\rVert_{\infty,K}\coloneq\lVert{}\cdot{}\!\mid\ell_\infty(K)\rVert$ and $\epsilon>0$ is small. Theorems \ref{thm:random_nn_barron} answers that question for $\alpha>1.$

Note that $\lVert h\rVert_1>0\Leftrightarrow\{h\neq0\}$ has positive measure $\Leftrightarrow\{h^\alpha\neq0\}$ has positive measure $\Leftrightarrow\lVert f\mid\sfE^\alpha\rVert=\lVert h^\alpha\rVert_{1,\infty}>0.$

If $\lVert f\mid\sfE^\alpha\rVert\in(0,\infty)$ we may write $f(x)=\iint\delta^{(-\alpha)}(\langle w\mid x\rangle-b)h(w,b)\,db\,\underline{dw}=$
\begin{align*}
\mathds{E}_{(w,b)\sim H}\Bigl(\lVert h\rVert_1\varphi(w,b)\delta^{(-\alpha)}(\langle w\mid x\rangle-b)\Bigr)
\end{align*}
where $\varphi=h\div\lvert h\rvert$ if $h\neq 0$ and $\varphi=1$ (or any other complex number with unit modulus) if $h=0,$ so a natural choice for the $a_j$ would be to take
\begin{align*}
a_j\coloneq \lVert h\rVert_1\varphi(w_j,b_j)/m
\end{align*}
i.e., $f-f_m=\mathds{E}(\xi)-\frac{1}{m}\sum_{j=1}^m\xi_j$ with $\xi$ being an i.i.d.\ copy of the $\xi_j$ given by
\begin{align*}
\xi_j(x)\coloneq\lVert h\rVert_1\varphi(w_j,b_j)\delta^{(-\alpha)}(\langle w_j\mid x\rangle-b_j)
\end{align*}
\begin{rem}
The reason we do not take the $(w_j,b_j)$ to be i.u.d.\ on $\sphere^{n-1}\times[-r,r]$ is because then $a_j=2r\lvert\sphere^{n-1}\rvert h^\alpha(w_j,b_j)/m,$ but $w\mapsto\lVert h^\alpha(w,\cdot)\rVert_\infty$ need not even be $\ell_2(\sphere^{n-1}),$ let alone $\ell_\infty(\sphere^{n-1}),$ and $\lVert h^\alpha(w_j,\cdot)\rVert_\infty$ possibly having infinite variance does not permit us to use the central limit theorem to get that
\begin{align*}
\text{``}\Bigl(\mathds{E}(\xi)-\frac{1}{m}\sum_{j=1}^m\xi_j\Bigr)\sqrt{m}\leadsto N(0,\mathds{V}(\xi))\text{''}
\end{align*}
where $\leadsto$ and $\mathds{V}$ denote convergence in law and the variance resp. In actuality the situation is a little more complicated, hence the quotes, but this is the high-level idea. By doing it the way we did, $\xi$ is bounded (see Lemma \ref{lem:xi_bound} below) so we can use the usual Chernoff machinery.
\end{rem}
Before stating the main theorem of this section, we need some preparation.
\begin{lem}\label{lem:xi_bound}
$\P$-almost everywhere $\lVert\xi_j\rVert_{\infty,K}\leq 2r\lVert f\mid \sfE^\alpha\rVert\delta^{(-\alpha)}(2r)$ if $f\in\sfE^\alpha$
\end{lem}
\begin{proof}
Since $\lvert\varphi\rvert=1$ and $\delta^{(-\alpha)}\geq 0$ is nondecreasing,
\begin{align*}
\lvert\xi_j(x)\rvert=\lVert h\rVert_1\lvert\varphi(w_j,b_j)\delta^{(-\alpha)}(\langle w_j\mid x\rangle-b_j)\rvert\leq\lVert h\rVert_1\delta^{(-\alpha)}(r+r)
\end{align*}
because $(w_j,b_j)\sim H$ implies that $\P$-almost everywhere $\lvert w_j\rvert=1$ and $\lvert b_j\rvert\leq r,$ ergo $x\in K(r)\Rightarrow\langle w_j\mid x\rangle\leq\lvert w_j\rvert.\lvert x\rvert\leq r.$ The desideratum now readily follows from the fact that $\lVert h\rVert_1\leq 2r\lVert h^\alpha\rVert_{1,\infty}=2r\lVert f\mid\sfE^\alpha\rVert.$
\end{proof}
Let $N_S^\delta$ be the $\delta$-covering number of $S\subset\R^n$ by Euclidean balls \cite[Def.\ 4.2.2]{vershynin}. A corollary of \cite[Thm 7.1.1]{boroczky} is that
\begin{prop}\label{prop:boroczky}
Let $S\subset\R^n$ be a convex compactum with nonempty interior. Then, for all $\delta\in(0,\rho),$ where $\rho$ is the maximal radius of any ball contained in $S,$
\begin{align*}
\theta_n\lvert S\rvert(1-\delta/\rho)^n\leq N_S^\delta\lvert K(\delta)\rvert\leq\theta_n\lvert S\rvert(1+\delta/\rho)^n
\end{align*}
and $\theta_n$ is a constant depending on $n$ only. Ergo, $N_S^\delta\lvert K(\delta)\rvert\approx\theta_n\lvert S\rvert$ for small $\delta.$
\end{prop}
It is known that $n\preccurlyeq\theta_n\preccurlyeq n\ln n$ with reasonable constants \cite[Thm 8.2.1]{boroczky}.

Handling $\alpha=1$ is trickier than $\alpha>1,$ so we deal with $\alpha>1$ first.

\fbox{$\alpha>1$}

The following general lemma (notation not matching with the rest of the paper) will be proved in Appendix \ref{sec:appendix}.
\begin{lem}\label{lem:general_sup_norm_prob_bound}
Let $K\subset\R^\lambda$ and $(\Omega,\Sigma,\P)$ be a probability space rich enough to house $m$ i.i.d.\ copies of $f,$ denoted $\{f_m\}_{m=1}^n,$ where $f:\Omega\times K\to\C$ is such that $f(\omega,\cdot)$ is $\ell$-Lipschitz continuous for almost every $\omega\in\Omega$ and $f(x)$ is a random variable bounded by $b>0$ for every $x\in K.$ Then, while $n\geq 4\lambda(b/\epsilon)^2$ for some $\epsilon>0$ we have that
\begin{align*}
\P\Bigl\{\bigl\lVert\mathds{E}(f)-\frac{1}{n}\sum_{m=1}^n f_m\bigr\rVert_{\infty,K}>\epsilon\Bigr\}\leq 2N_K^\delta\delta^\lambda(\ell\sqrt{e}/\lambda)^\lambda\zeta^\lambda e^{-(\epsilon/4)\zeta}
\end{align*}
where $\delta,\zeta$ are defined by
\begin{align*}
\ell\delta=\frac{\lambda}{\zeta}\text{ with }\zeta=(\epsilon n/b^2)\biggl(1+\sqrt{\displaystyle 1-4\lambda(b/\epsilon)^2/n}\biggr)
\end{align*}
Note that $\zeta\approx 2\epsilon n/b^2$ for large $n,$ so in this r{\'e}gime we have that\begin{align*}
\P\Bigl\{\bigl\lVert\mathds{E}(f)-\frac{1}{n}\sum_{m=1}^n f_m\bigr\rVert_{\infty,K}>\epsilon\Bigr\}\lessapprox 2\lvert K\rvert\Theta_\lambda(2\pi\lambda)^{-\lambda/2}(\epsilon\ell n/b^2)^\lambda\exp(-(n/2)(\epsilon/b)^2)
\end{align*}
where $\Theta_\lambda\coloneq\theta_\lambda\sqrt{\pi}(\lambda^3+\lambda^2+\tfrac{1}{2}\lambda+\tfrac{1}{30})^{1/6}$
\end{lem}
\begin{rem}
By the triangle inequality, $\mathds{E}(f)$ is also $\ell$-Lipschitz, so $K$ being separable (as a subset of a separable metric space) yields that
\begin{center}
$\bigl\lVert\mathds{E}(f)-\frac{1}{n}\sum_{m=1}^n f_m\bigr\rVert_{\infty,K}$ is really a countable supremum and thus measurable.
\end{center}
\end{rem}
We are now ready to prove the main theorem of this section for $\alpha>1.$
\begin{theorem}[$\alpha>1$]\label{thm:random_nn_barron}
Let $f\in\sfE^\alpha$ with $\lVert f\mid\sfE^\alpha\rVert>0.$ Then, for large $m,$
\begin{align*}
\P\Bigl\{\lVert f-f_m\rVert_{\infty,K}>\epsilon\Bigr\}&\lessapprox 2\lvert K(\tfrac{\alpha-1}{2})\rvert\Theta_n(2\pi n)^{-n/2}(\epsilon m/\Lambda)^n\exp(-(m/2)(\epsilon/\Lambda)^2)
\end{align*}
provided $\Lambda\sqrt{4n/m}\leq\epsilon,$ where $\Lambda\coloneq 2r\lVert f\mid \sfE^\alpha\rVert\delta^{(-\alpha)}(2r)$ is the bound from Lemma \ref{lem:xi_bound}. To get a more insightful bound, we may set
\begin{align}\label{eq:geq4rem}
\epsilon=\Lambda\sqrt{\frac{kn\ln m}{m}}\text{ with }k\ln m\geq 4
\end{align}
which, after simplfying, yields that
\begin{align*}
\P\Biggl\{\lVert f-f_m\rVert_{\infty,K}>\Lambda\sqrt{\frac{kn\ln m}{m}}\Biggr\}\lessapprox 2\lvert K(\tfrac{\alpha-1}{2})\rvert\Theta_n\biggl(\frac{k\ln m}{\displaystyle 2\pi m^{k-1}}\biggr)^{\!\!n/2}
\end{align*}
for $m\gg e^{4/k}$ (see subsequent remark).
\end{theorem}
\begin{rem}
We need $k\ln m\geq 4$ in \eqref{eq:geq4rem} because $n\geq 4\lambda(b/\epsilon)^2$ in Lemma \ref{lem:general_sup_norm_prob_bound}, but as we want $m$ large and $k>1$ this is not a problem. Taking $\epsilon$ as in \eqref{eq:geq4rem} means that $4\lambda(b/\epsilon)^2/n$ corresponds to $4/(k\ln m),$ which still vanishes in the limit as $m\to\infty,$ hence the analogue of $\zeta\approx 2\epsilon n/b^2$ for large $n$ still applies.
\end{rem}
\begin{proof}
By Lemma \ref{lem:xi_bound} $\xi(x)$ is $\P$-almost everywhere bounded for all $x\in K.$ To bound the Lipschitz constant of $\xi$ it suffices to uniformly bound $\lvert\nabla\xi\rvert.$ Letting $(w,b)\sim H,$
\begin{align*}
\nabla\xi(x)=\lVert h\rVert_1\varphi(w,b)\delta^{(1-\alpha)}(\langle w\mid x\rangle-b)w
\end{align*}
so it readily follows that $\P$-almost everywhere $\xi$ is Lipschitz continuous on $K$ with Lipschitz constant $2r\lVert f\mid\sfE^\alpha\rVert\delta^{(1-\alpha)}(2r).$ Accordingly, by Lemma \ref{lem:general_sup_norm_prob_bound}, after simpli- fying a bit,
\begin{align*}
\P\Bigl\{\lVert f-f_m\rVert_{\infty,K}>\epsilon\Bigr\}&\lessapprox 2\lvert K\rvert\Theta_n(2\pi n)^{-n/2}\biggl(\frac{\alpha-1}{2r}\biggr)^{\!\!n}(\epsilon m/\Lambda)^n\exp(-(m/2)(\epsilon/\Lambda)^2)
\end{align*}
provided $\Lambda\sqrt{4n/m}\leq\epsilon.$ The desideratum now follows by simplifying some more.
\end{proof}
\newpage
Theorem \ref{thm:random_nn_barron} shows that the accuracy rate $m^{-1/2}$ (up to log factors) obtained in \cite{barron} also applies to random neural networks with activation function $\delta^{(-\alpha)}$ whilst $\alpha>1.$ Indeed, the probabilities in Theorem \ref{thm:random_nn_barron} will be $<1$ for large enough $m,$ whence there must be realizations of $f_m$ such that $\lVert f-f_m\rVert_{\infty,K}\leq\epsilon$ if $m$ is large enough.

\fbox{$\alpha=1$}

Since $\delta^{(-1)}$ is not (Lipschitz) continuous, we cannot directly apply Lemma \ref{lem:general_sup_norm_prob_bound} in this case. Unfortunately the author ran out of time to work on this case. We did obtain two lemmas which may be helpful in working on the problem:

\begin{lem}\label{lem:N^1_lip}
Any $f\in\sfE^1$ is $\lVert f\mid\sfE^1\rVert$-Lipschitz continuous.
\end{lem}
\begin{proof}
By Definition \ref{def:n_alpha}, $\exists h^1\in\ell_{1,\infty}(-\infty)$ such that
\begin{align*}
f(x)=\mathds{E}_{w,b}[2r\lvert\sphere^{n-1}\rvert h^1(w,b)\delta^{(-1)}(\langle w\mid x\rangle-b)]
\end{align*}
Consequently,
\begin{align*}
\lvert f(x)-f(u)\rvert\leq 2r\lvert\sphere^{n-1}\rvert\cdot\mathds{E}_{w,b}\Bigl[\lvert h^1(w,b)\rvert.\lvert\delta^{(-1)}(\langle w\mid x\rangle-b)-\delta^{(-1)}(\langle w\mid u\rangle-b)\rvert\Bigr]
\end{align*}
Now, $\lvert\delta^{(-1)}(\langle w\mid x\rangle-b)-\delta^{(-1)}(\langle w\mid u\rangle-b)\rvert$ is exactly the indicator of the event that $x,u$ lie on different sides of the random hyperplane $x\mapsto\langle w\mid x\rangle-b,$ so from the proof of Theorem V in \cite{sampta2025} we know that (taking $p,r\to 0$ there)
\begin{align*}
\mathds{E}_b\Bigl[\lvert\delta^{(-1)}(\langle w\mid x\rangle-b)-\delta^{(-1)}(\langle w\mid u\rangle-b)\rvert\Bigr]\leq\frac{\lvert x-u\rvert}{2r}
\end{align*}
so $\lvert f(x)-f(u)\rvert\leq\lvert x-u\rvert\int\lVert h^1(w,\cdot)\rVert_\infty\,\underline{dw}=\lVert h^1\rVert_{1,\infty}\lvert x-u\rvert.$
\end{proof}

\begin{lem}[$\alpha=1$]\label{lem:sup_measurable}
$\lVert f-f_m\rVert_{\infty,K}$ is measurable.
\end{lem}
\begin{proof}
While $\delta^{(-1)}(b)=1$ if $b>0$ and $=0$ if $b<0,$ we never specified $\delta^{(-1)}(0).$ Hence, let $\delta^{(-1)}_c(0)=c.$ It then follows that
\begin{align*}
\lVert f-f_m\rVert_{\infty,K}=\max_{J\subseteq[m]}\lVert f-\sum_{j=1}^m a_j\delta^{(-1)}_{[J\ni j]}(\langle w_j\mid\!{}\cdot{}\rangle-b_j)\rVert_{\infty,K}
\end{align*}
As $\R^n$ is separable, so is $K$ being a subset of a separable metric space, hence the desideratum follows because $\textstyle\sum_{j=1}^m\displaystyle a_j\delta^{(-1)}_{[J\ni j]}(\langle w_j\mid\!{}\cdot{}\rangle-b_j)$ is everywhere continuous along a sector and $f$ is continuous pursuant to Lemma \ref{lem:N^1_lip}.
\end{proof}

\subsection{Discussion}\label{sec:main_discussion}

Lemma \ref{lem:dh_decomp} and Corollary \ref{cor:d_decomp} suggest that $\text{\DH}_\ell^\alpha$ and $\text{D}_\ell^\alpha$ combine two different kinds of smoothness; Theorems \ref{thm:main} \& \ref{thm:random_nn_barron} partly substantiate this. Looking at $\text{D}_\ell^\alpha=\partial^\alpha\text{D}_\ell,$ the proof of Theorem \ref{thm:main} suggests that $\partial^\alpha$ manifests itself via the activation function $\delta^{(-\alpha)}.$ The substantiation is only partial though, since the $(n-1)/2$ ``units of smoothness'' from $\text{D}_\ell^0$ (cf.\ \eqref{eq:n_asymp}) do not quite jibe with the accuracy rate of $m^{-1/2}\ln m$ we got from \eqref{eq:geq4rem}; for according to the curse of dimensionality, a piece of mathematical folklore which asserts that accuracy \\ rates are often of the form
\begin{align*}
(\text{number of terms})^{-(\text{smoothness})/(\text{dimension})}
\end{align*}
we should, with $(n-1)/2$ units of smoothness, expect an accuracy rate of
\begin{align*}
m^{-((n-1)/2)/n}=m^{1/(2n)-1/2}
\end{align*}
which is slightly worse than the accuracy rate we are able to obtain. This might lead one to believe that there is perhaps a slightly smoother version of $\sfL^{\!\alpha}$ which subsumes $\B^\alpha$ (at least for $\alpha>1)$ and would provide \emph{the} nonstandard notion of smoothness which sees Barron space conform with the curse of dimensionality. Perhaps our ``lifting'' of $\B\subset\sfL$ to $\B^\alpha\subset\sfL^{\!\alpha}$ was too na{\"i}ve?

Actually, \cite[Cor.\ 6.10]{pinkus} offers some evidence to the contrary, namely that $\tfrac{n-1}{2}$ \emph{is} the right amount of smoothness to attribute to the $\B^\alpha$ spaces. Per \cite[Cor.\ 6.10]{pinkus} $\inf\{\lVert f-f_m\rVert_2:f_m\}\lesssim m^{-s/n}$ where $f$ has square integrable weak derivatives of orders $0,\ldots,s\leq\tfrac{n-1}{2}+\alpha$ and $f_m$ is of the form
\begin{align*}
\R^n\ni x\mapsto\sum_{j=1}^m a_j\delta^{(-\alpha)}(\langle w_j\mid x\rangle-b_j)\qquad(a_j\in\C,b_j\in\R,w_j\in\R^n)
\end{align*}
Note that $\alpha\in\N$ is one more than the smoothness of $\delta^{(-\alpha)}.$

The first similarity between $\B^\alpha\subset\sfL^{\!\alpha}$ and \cite[Cor.\ 6.10]{pinkus} is that the $\tfrac{n-1}{2}+\alpha$ in the statement of \cite[Cor.\ 6.10]{pinkus} is exactly the total smoothness of $\text{D}_\ell^\alpha,$ which has $\tfrac{n-1}{2}$ units of nonclassical smoothness and $\alpha$ units of classical smoothness. The second similarity is that $f$ having square integrable weak derivatives up to order $\alpha$ is in \\ a very real sense the $\ell_2$ analogue of the $\B^\alpha$ space. Indeed,
\begin{align*}
\F\{\lvert\cdot\rvert^2\phi\}=-\Delta\F\{\phi\}\Rightarrow\B^\alpha=(-\Delta)^{\alpha/2}\F\{\ell_1\}\cap\F\{\ell_1\}
\end{align*}
so the $\ell_2$ analogue of $\B^\alpha\subset\F\{\ell_1\}\subset\ell_\infty$ would be
\begin{align*}
(-\Delta)^{\alpha/2}\F\{\ell_2\}\cap\F\{\ell_2\}=(-\Delta)^{\alpha/2}\ell_2\cap\ell_2
\end{align*}
because 1 is the H{\"o}lder conjugate of $\infty$ (to get a space $\subset\ell_\infty$ we need to put $\ell_1$ inside the Fourier transforms) and 2 is its own H{\"o}lder conjugate (we wanted a space $\subset\ell_2).$ That $(-\Delta)^{\alpha/2}\ell_2\cap\ell_2$ is exactly the space of functions with square integrable weak derivatives up to order $\alpha$ has long since been known; cf., e.g., Remark 3 in {\S\,}2.2.2 of \cite{triebel}.

Whether more evidence may be found substantiating that the $\B^\alpha$ spaces do (not) \\ in fact defy the curse of dimensionality, albeit with nonstandard smoothness, is, \\ if you ask the author, a very interesting direction for future research.

We now discuss how our results relate to those in \cite{radonNN,felix,liflyand} in detail.

If we insist $\text{D}_\ell$ has $\tfrac{n-1}{2}$ units of nonstandard smoothness, then $\B$ has $\tfrac{n-1}{2}$ units of nonstandard smoothness too (since $\B\subset\sfL).$ The result that motivated this paper, \cite[Thm 1]{liflyand}, substantiates this; in a nutshell, it is shown that $t^\alpha\partial^\alpha f_0$ is a conti- nuous function on $(0,\infty)$ vanishing at the endpoints for all $\alpha\in$ $[0,\tfrac{n-1}{2}]$ whenever $f\in\B.$ Although (Corollary \ref{cor:t_alpha_partial_alpha})
\begin{align*}
\iota t^\alpha\partial^\alpha=\M_0^{-1}\biggl\{\frac{\Gamma(1-iy)}{\Gamma(1-iy-\alpha)}\M_0\{\cdot\}\biggr\}
\end{align*}
differs from $\text{D}_0=\M_0^{-1}\{N_0\M_0\{\cdot\}\}=$
\begin{align*}
\M_0^{-1}\biggl\{\frac{\displaystyle\Gamma(\tfrac{n-iy}{2})}{2\displaystyle\pi^{(n-1)/2}\Gamma(\tfrac{1-iy}{2})}\M_0\{\cdot\}\biggr\}=\tfrac{1}{2}\M_0^{-1}\biggl\{\frac{\displaystyle\Gamma(\tfrac{n-1}{2}+\tfrac{1-iy}{2})}{\displaystyle\pi^{(n-1)/2}\Gamma(\tfrac{1-iy}{2})}\M_0\{\cdot\}\biggr\}
\end{align*}
the fact that (cf.\ \eqref{eq:n_asymp}) as $\lvert y\rvert\to\infty$
\begin{center}
$\biggl\lvert\frac{\Gamma(1-iy)}{\Gamma(1-iy-\alpha)}\biggr\rvert\sim\lvert y\rvert^\alpha$ and $\biggl\lvert\frac{\displaystyle\Gamma(\tfrac{n-1}{2}+\tfrac{1-iy}{2})}{\displaystyle\pi^{(n-1)/2}\Gamma(\tfrac{1-iy}{2})}\biggr\rvert\sim\biggl\{\frac{\lvert y\rvert}{2\pi}\biggr\}^{\!(n-1)/2}$
\end{center}
have such similar asymptotic behavior if $\alpha=\tfrac{n-1}{2}$ is surely no coincidence.

We also showed that $\sfL\subset\sfE;$ more precisely, we showed that if $f\in\sfL$ has $\tfrac{n-1}{2}$ units of atypical smoothness and $h=\scrA\{\scrD\{\scrZ\{f\}\}\}$, then $f=\dualradon\{h\}.$ In other words, $h=\scrA\{\scrD\{\scrZ\{\dualradon\{h\}\}\}\};$ this relation is redolent of $\varphi\propto_n\!\mathcal{R}\{(-\Delta)^{(n-1)/2}\dualradon\{\varphi\}\},$ where $\varphi$ is an even Schwartz function on $\sphere^{n-1}\times\R$ and $\mathcal{R}$ is the Radon transform \cite[Lemma 2]{radonNN}.

In \cite{radonNN}, like herein, infinitely wide shallow neural networks are assimilated to the dual Radon transform, except there only the ReLU $(\delta^{(-2)})$ activation function is considered and infinitely wide shallow ReLU networks are defined to be functions \\ of the form
\begin{align*}
\R^n\ni x\mapsto\iint\Bigl(\delta^{(-2)}(\langle w\mid x\rangle-b)-\delta^{(-2)}(-b)\Bigr)\alpha(\underline{dw},db)
\end{align*}
so that the integrand is bounded, where $\alpha$ is any signed Radon measure on $\sphere^{n-1}\times\R$ with finite total variation.

Lastly, Theorem \ref{thm:main} warrants comparison with \cite[{\S\,}7]{felix}. In our notation, \cite[{\S\,}7]{felix} investigates the relations between the $\B^\alpha$ spaces with a slightly different norm and $\sfE(\delta^{(-1)})$ \& $\sfE(\delta^{(-2)}).$ We recall that $\sfE(\phi)$ (defined in {\S\,}\ref{sec:intro}) is the space of functions
$f:U\to\R$ (with $U\subset\R^n$ nonempty) of the form
\begin{center}
$f(x)=\textstyle\int\displaystyle a\cdot\phi(\langle w\mid x\rangle+c)\,\mu(da,dw,dc)$ where $\phi:\R\to\R$ \\ and $\mu$ is a Borel probability measure on $\R\times\R^n\times\R$
\end{center}
Despite the differences between $\sfE^\alpha$ and $\sfE(\delta^{(-\alpha)}),$ their results substantiate ours:

In \cite[Lemma 7.1]{felix} it is shown that $\B^\alpha\subset\sfE(\delta^{(-\alpha)})$ for $\alpha=1,2$ if $U$ is bounded; it \\ is also shown that $\sfE(\delta^{(-2)})\subset\sfE(\delta^{(-1)})$ if $U$ is bounded, which, together with the conclusion of \cite[Prop.\ 7.4]{felix}, i.e., $\B^1\not\subset\sfE(\delta^{(-2)})$ if $U$ has nonempty interior, makes for a very interesting pair of results whose adaptions to the framework herein are an interesting direction for future research as well.

{%
\appendix

\section{Barron functions need not be very smooth}\label{sec:barron_example}

Let $\sigma_0:\R\ni u\mapsto\bigl[\lvert u\rvert>1\bigr]\div\lvert u\rvert^{n+1}$ be the radial part of $\sigma;$ i.e., $\sigma=\sigma_0\circ\lvert\cdot\rvert.$

The Fourier transform of $\sigma$ may be Lipschitz continuous, but that is all it is:
\begin{theorem*}
$\F\{\sigma\}(x)=\lvert\sphere^{n-1}\rvert\frac{\pi\lvert x\rvert}{n-1}+G(x)$ where $\Delta G$ is continuous
\end{theorem*}
\begin{proof}
Since $\sigma$ is radial, so is $\F\{\sigma\}.$ Moreover, \cite[{\S\,}3.3]{ehm} 
\begin{align*}
\F\{\sigma\}(x)&=\lvert\sphere^{n-1}\rvert\int\cos(t\lvert x\rvert)\int_{\lvert t\rvert}^\infty\sigma_0(u)u(u^2-t^2)^{(n-3)/2}\,du\,dt\\
&=2\lvert\sphere^{n-1}\rvert\int_0^\infty\cos(t\lvert x\rvert)\int_t^\infty\sigma_0(u)u(u^2-t^2)^{(n-3)/2}\,du\,dt\\
&\overset{\text{F}}{=}2\lvert\sphere^{n-1}\rvert\int_0^\infty\sigma_0(u)u\int_0^u\cos(t\lvert x\rvert)(u^2-t^2)^{(n-3)/2}\,dt\,du
\end{align*}
Substituting $dt=u\,dv$ yields that
\begin{align*}
\F\{\sigma\}(x)&=2\lvert\sphere^{n-1}\rvert\int_0^\infty\sigma_0(u)u^{n-1}\int_0^1\cos(uv\lvert x\rvert)(1-v^2)^{(n-3)/2}\,dv\,du\\
&\overset{\text{F}}{=}2\lvert\sphere^{n-1}\rvert\int_0^1(1-v^2)^{(n-3)/2}\int_0^\infty\sigma_0(u)u^{n-1}\cos(uv\lvert x\rvert)\,du\,dv\\
&=2\lvert\sphere^{n-1}\rvert\int_0^1(1-v^2)^{(n-3)/2}\int_1^\infty u^{-2}\cos(uv\lvert x\rvert)\,du\,dv
\end{align*}
Let us investigate $\textstyle\int_1^\infty\displaystyle u^{-2}\cos(uz)\,du.$
\begin{align*}
\int_1^\infty u^{-2}\cos(uz)\,du=\int_1^\infty u^{-2}\cos(u\lvert z\rvert)\,du=\lvert z\rvert\int_{\lvert z\rvert}^\infty w^{-2}\cos(w)\,dw
\end{align*}
where $dw=\lvert z\rvert\,du.$ Integration by parts now yields that
\begin{align*}
\int_1^\infty u^{-2}\cos(uz)\,du&=-\lvert z\rvert\Bigl[\cos(w)/w\Bigr]_{\lvert z\rvert}^\infty-\lvert z\rvert\int_{\lvert z\rvert}^\infty\frac{\sin(w)}{w}\,dw\\
&=\cos(z)-\lvert z\rvert\lim_{m\to\infty}\int_{\lvert z\rvert}^m\frac{\sin(w)}{w}\,dw
\end{align*}
By definition of the sine integral $\Si$ \cite[{\S\,}6.2(ii)]{dlmf}:
\begin{align*}
\int_{\lvert z\rvert}^m\frac{\sin(w)}{w}\,dw=\int_0^m\frac{\sin(w)}{w}\,dw-\int_0^{\lvert z\rvert}\frac{\sin(w)}{w}\,dw=\Si(m)-\Si(\lvert z\rvert)
\end{align*}
Since Si is odd and $\Si(\infty)=\pi/2,$ plugging back in yields that
\begin{align*}
\int_1^\infty u^{-2}\cos(uz)\,du=\cos(z)-\tfrac{\pi}{2}\lvert z\rvert+z\Si(z)
\end{align*}
All in all, we therefore find that
\begin{align*}
\F\{\sigma\}(x)&=2\lvert\sphere^{n-1}\rvert\int_0^1(1-v^2)^{(n-3)/2}\Bigl(\cos(v\lvert x\rvert)+v\lvert x\rvert\Si(v\lvert x\rvert)-\tfrac{\pi}{2}v\lvert x\rvert\Bigr)\,dv\\
&=2\lvert\sphere^{n-1}\rvert\int_0^1(1-v^2)^{(n-3)/2}\Bigl(\cos(v\lvert x\rvert)+v\lvert x\rvert\Si(v\lvert x\rvert)\Bigr)\,dv\\
&-\pi\lvert\sphere^{n-1}\rvert.\lvert x\rvert\int_0^1 v(1-v^2)^{(n-3)/2}\,dv
\end{align*}
Note that $\textstyle\int_0^1\displaystyle v(1-v^2)^{(n-3)/2}\,dv=1/(n-1)$ (substitute $y=1-v^2).$ Letting $G(x)$ be the minuend, all that remains to be shown is that $\Delta G$ is continuous. Since $G$ \\ is radial, $\Delta G(x)=\partial_\circ^2 G(x)+(n-1)\partial_\circ G(x)/\lvert x\rvert,$ where $\partial_\circ$ denotes the derivative w.r.t.\ $\lvert x\rvert$ \cite[Lemma 1.4.1]{dai}; using this formula it is easy to check that
\begin{align*}
\Delta G(x)=2\lvert\sphere^{n-1}\rvert\int_0^1 v^2(1-v^2)^{(n-3)/2}\biggl(\frac{\sin(v\lvert x\rvert)}{v\lvert x\rvert}+(n-1)\frac{\Si(v\lvert x\rvert)}{v\lvert x\rvert}\biggr)\,dv
\end{align*}
where the exchange of differentiation and integration is justified by bounded convergence. Since sin and Si are odd, entire functions, $\partial G$ is (sequentially) continuous by bounded convergence again.
\end{proof}

\section{Proof of Lemma \ref{lem:NOM}}\label{sec:NOM}

The idea is to show that $N_\ell^\alpha\in O_C\subset O_M$ \cite[pg.\ 441]{horvath}; now, $N\in O_C$ iff $N$ is the product of a polynomial and a smooth function all of whose derivatives vanish at infinity \cite[Example 2.12.9]{horvath}. Illustrating this for all $n\geq2$ and $\ell\in\N_0$ would soon get cumbersome in terms of notation, so we instead sketch the proof for $N_5^\alpha$ with $n=3$ and $n=2.$

If $n=3,$ then $N_5^\alpha(t)$ is proportional to
\begin{align*}
\frac{\displaystyle\Gamma(\tfrac{8-it}{2})(-1)^\alpha(it)_\alpha}{\displaystyle\Gamma(\tfrac{2-it}{2})(\tfrac{1+it}{2})_2}
\end{align*}
Since $\Gamma(\tfrac{8-it}{2})=(\tfrac{2-it}{2})_3\Gamma(\tfrac{2-it}{2}),$ taking $(-1)^\alpha(it)^\alpha(\tfrac{2-it}{2})_3$ as the polynomial, we need only show $1/(\tfrac{1+it}{2})_2$ is a smooth function all of whose derivatives vanish at infinity. Plainly functions of the form $\R\ni t\mapsto 2/(\sigma+it)$ are smooth functions all of whose derivatives vanish at infinity, and the product of two smooth functions all of whose derivatives vanish at infinity is patently again a smooth function all of whose deri- vatives vanish at infinity.

If $n=2,$
then $N_5^\alpha(t)$ is proportional to
\begin{align*}
\frac{\displaystyle\Gamma(\tfrac{7-it}{2})(-1)^\alpha(it)_\alpha}{\displaystyle\Gamma(\tfrac{2-it}{2})(\tfrac{1+it}{2})_2}
\end{align*}
Taking $(-1)^\alpha(it)^\alpha(\tfrac{2-it}{2})_3$ as the polynomial again and remembering that the product of two smooth functions all of whose derivatives vanish at infinity is patently again a smooth function all of whose derivatives vanish at infinity, it suffices to show $\Gamma(\tfrac{7-it}{2})\div\Gamma(\tfrac{8-it}{2})$ is a smooth function all of whose derivatives vanish at infinity.

It follows pretty directly from the definitions of the di- and polygamma functions that any derivative of $\Gamma(t)$ is of the form $\Gamma(t)$ times a polynomial in the di- and polygamma functions. Interestingly, a strikingly similar statement is true for $1\div\Gamma(t);$ every derivative of $1\div\Gamma(t)$ is of the form $1\div\Gamma(t)$ times a (different) polynomial in the di- and polygamma functions. So, when applying the product rule to $\Gamma(\tfrac{7-it}{2})\div\Gamma(\tfrac{8-it}{2})$ (repeatedly), we get $\Gamma(\tfrac{7-it}{2})\div\Gamma(\tfrac{8-it}{2})$ times a polynomial in the di- and polygamma functions with arguments $\tfrac{7-it}{2}$ and $\tfrac{8-it}{2}.$

Now, the digamma function of $\tfrac{7-it}{2}$ or $\tfrac{8-it}{2}$ grows logarithmically in $\lvert t\rvert$ as $\lvert t\rvert\to\infty$ \cite[(5.11.2)]{dlmf} and the polygamma functions vanish at infinity, so the polynomials \\ in the di- and polygamma functions with arguments $\tfrac{7-it}{2}$ and $\tfrac{8-it}{2}$ grow poly- logarithmically in $\lvert t\rvert$ as $\lvert t\rvert\to\infty.$ Fortunately,
\begin{align}\label{eq:gamma_quotient_asymp}
\lvert \Gamma(\tfrac{7-it}{2})\div\Gamma(\tfrac{8-it}{2})\rvert\sim\sqrt{2/\lvert t\rvert}\qquad\text{as }\lvert t\rvert\to\infty
\end{align}
thus $\Gamma(\tfrac{7-it}{2})\div\Gamma(\tfrac{8-it}{2})$ is a smooth function all of whose derivatives vanish at infinity. Note that \eqref{eq:gamma_quotient_asymp}, like \eqref{eq:n_asymp}, follows from the fact that
\begin{align*}
\lvert\Gamma(\sigma+it)\rvert\sim\sqrt{2\pi}e^{-\pi\lvert t\rvert/2}\lvert t\rvert^{\sigma-1/2}\qquad\text{uniformly as }\lvert t\rvert\to\infty
\end{align*}
for $\sigma$ in any compact subset of $\R$ \cite[pg.\ 244]{pribitkin}.

\section{Proof of Lemma \ref{lem:general_sup_norm_prob_bound}}\label{sec:appendix}

We repeat Lemma \ref{lem:general_sup_norm_prob_bound} (and subsequent remark) here for convenience.

\begin{lem}
Let $K\subset\R^\lambda$ be a convex compactum with nonempty interior and $(\Omega,\Sigma,\P)$ be a probability space rich enough to house $m$ i.i.d.\ copies of $f,$ denoted $\{f_m\}_{m=1}^n,$ where $f:\Omega\times K\to\C$ is such that $f(\omega,\cdot)$ is $k$-Lipschitz continuous for almost every $\omega\in\Omega$ and $f(x)$ is a random variable bounded by $b>0$ for every $x\in K.$ Then, while $n\geq 4\lambda(b/\epsilon)^2$ for some $\epsilon>0$ we have that
\begin{align*}
\P\Bigl\{\bigl\lVert\mathds{E}(f)-\frac{1}{n}\sum_{m=1}^n f_m\bigr\rVert_{\infty,K}>\epsilon\Bigr\}\leq 2N_K^\delta\delta^\lambda(k\sqrt{e}/\lambda)^\lambda\zeta^\lambda e^{-(\epsilon/4)\zeta}
\end{align*}
where $\delta,\zeta$ are defined by
\begin{align}
k\delta=\frac{\lambda}{\zeta}\text{ with }\zeta=(\epsilon n/b^2)\biggl(1+\sqrt{\displaystyle 1-4\lambda(b/\epsilon)^2/n}\biggr)\label{eq:kdelta}
\end{align}
Note that $\zeta\approx 2\epsilon n/b^2$ for large $n,$ so in this r{\'e}gime we have that
\begin{align*}
\P\Bigl\{\bigl\lVert\mathds{E}(f)-\frac{1}{n}\sum_{m=1}^n f_m\bigr\rVert_{\infty,K}>\epsilon\Bigr\}\lessapprox 2\lvert K\rvert\Theta_\lambda(2\pi\lambda)^{-\lambda/2}(\epsilon kn/b^2)^\lambda\exp(-(n/2)(\epsilon/b)^2)
\end{align*}
where $\Theta_\lambda\coloneq\theta_\lambda\sqrt{\pi}(\lambda^3+\lambda^2+\tfrac{1}{2}\lambda+\tfrac{1}{30})^{1/6}$
\end{lem}
\begin{rem}
By the triangle inequality, $\mathds{E}(f)$ is also $k$-Lipschitz, so $K$ being separable (as a subset of a separable metric space) yields that
\begin{center}
$\bigl\lVert\mathds{E}(f)-\frac{1}{n}\sum_{m=1}^n f_m\bigr\rVert_{\infty,K}$ is really a countable supremum and thus measurable.
\end{center}
\end{rem}
Let $B_z^\delta$ be the (separable) metric ball in the metric space $K$ with centre $z$ and radius $\delta.$
\begin{proof}
Let $N$ be a minimal $\delta$-net of $K$ with corresponding covering number $N_K^\delta.$ Then
\begin{align*}
\P\Bigl\{\bigl\lVert\mathds{E}(f)-\frac{1}{n}\sum_{m=1}^n f_m\bigr\rVert_{\infty,K}>\epsilon\Bigr\}\leq\sum_{z\in N}\P\Bigl\{\bigl\lVert\mathds{E}(f)-\frac{1}{n}\sum_{m=1}^n f_m\bigr\rVert_{\infty,B_z^\delta}>\epsilon\Bigr\}
\end{align*}
Since $\mathds{E}(f)-\frac{1}{n}\sum_{m=1}^n f_m$ is $2k$-Lipschitz,
\begin{align*}
\bigl\lvert\mathds{E}(f(z))-\frac{1}{n}\sum_{m=1}^n f_m(z)\bigr\rvert\geq\bigl\lvert\mathds{E}(f(x))-\frac{1}{n}\sum_{m=1}^n f_m(x)\bigr\rvert-2k\delta
\end{align*}
for all $x\in B_z^\delta$ and $z\in N,$ so for all $z\in N,$
\begin{align*}
\P\Bigl\{\bigl\lVert\mathds{E}(f)-\frac{1}{n}\sum_{m=1}^n f_m\bigr\rVert_{\infty,B_z^\delta}>\epsilon\Bigr\}\leq\P\Bigl\{\bigl\lvert\mathds{E}(f(z))-\frac{1}{n}\sum_{m=1}^n f_m(z)\bigr\rvert>\epsilon-2k\delta\Bigr\}
\end{align*}
Since $n\geq 4\lambda(b/\epsilon)^2\Rightarrow k\delta\leq\epsilon/4<\epsilon/2,$ Hoeffding's inequality yields that
\begin{align*}
\P\Bigl\{\bigl\lvert\mathds{E}(f(z))-\frac{1}{n}\sum_{m=1}^n f_m(z)\bigr\rvert>\epsilon-2k\delta\Bigr\}\leq 2\exp(-\tfrac{1}{2}(n/b^2)(\epsilon-2k\delta)^2)
\end{align*}
so all in all we have that
\begin{align*}
\P\Bigl\{\bigl\lVert\mathds{E}(f)-\frac{1}{n}\sum_{m=1}^n f_m\bigr\rVert_{\infty,K}>\epsilon\Bigr\}\leq 2N_K^\delta(k\delta)^\lambda(k\delta)^{-\lambda}\exp(-2(n/b^2)(\epsilon/2-k\delta)^2)
\end{align*}
One can verify that \eqref{eq:kdelta} locally minimizes $(k\delta)^{-\lambda}\exp(-2(n/b^2)(\epsilon/2-k\delta)^2)$ at
\begin{align*}
(\sqrt{e}/\lambda)^\lambda \zeta^\lambda e^{-(\epsilon/4)\zeta}
\end{align*}
which readily yields that $\P\Bigl\{\bigl\lVert\mathds{E}(f)-\frac{1}{n}\sum_{m=1}^n f_m\bigr\rVert_{\infty,K}>\epsilon\Bigr\}\lessapprox$
\begin{align*}
2N_K^\delta\lvert K(\delta)\rvert\frac{\displaystyle(\sqrt{e}/\lambda)^\lambda}{\lvert K(1)\rvert}(\epsilon kn/b^2)^\lambda\exp(-(n/2)(\epsilon/b)^2)
\end{align*}
Now, from \cite{alzer} we know that
\begin{align*}
\frac{1}{\lvert K(1)\rvert}=\pi^{-\lambda/2}\Gamma(\tfrac{\lambda}{2}+1)<\sqrt{\pi}(\lambda^3+\lambda^2+\tfrac{1}{2}\lambda+\tfrac{1}{30})^{1/6}\biggl(\frac{\lambda}{2\pi e}\biggr)^{\!\!\lambda/2}
\end{align*}
so, by Proposition \ref{prop:boroczky}, for large $n,$
\begin{align*}
N_K^\delta\lvert K(\delta)\rvert\frac{\displaystyle(\sqrt{e}/\lambda)^\lambda}{\lvert K(1)\rvert}\lessapprox\lvert K\rvert\Theta_\lambda(2\pi\lambda)^{-\lambda/2}\tag*{\qedhere}
\end{align*}
\end{proof}
}%

\end{document}